\documentclass{amsart}
\usepackage{graphicx}
\usepackage{amssymb}

\newtheorem{theorem}{Theorem}[section]
\newtheorem{proposition}[theorem]{Proposition}

\newtheorem{lemma}[theorem]{Lemma}

\theoremstyle{definition}

\title{Diophantine type of interval exchange maps}
\author{Dong Han Kim}
\address{Department of Mathematics Education, Dongguk University-Seoul, Seoul 100-715, Korea.}
\email{kim2010@dongguk.edu}
\thanks{Supported by the Korea Research Foundation(KRF) grant funded by the Korea government(MEST) (No. 2009-0068804).}


\begin{document}

\begin{abstract} Roth type irrational rotation numbers have several equivalent 
arithmetical characterizations as well as several equivalent characterizations in terms of the 
dynamics of the corresponding circle rotations.
In this paper we investigate how to generalize Roth-like Diophantine conditions to 
interval exchange maps. If one considers the dynamics in parameter space one can introduce two 
nonequivalent Roth-type conditions, the first (condition (Z))  by means of the Zorich cocyle \cite{Z1}, the 
second (condition (A)) by means of a further acceleration of the continued fraction algorithm  by Marmi-Moussa-Yoccoz introduced in \cite{MMY1}.
A third very natural condition (condition (D)) arises by considering the distance between the discontinuity 
points of the iterates of the map.
If one considers the dynamics of an interval exchange map in phase space then one can introduce the notion of 
Diophantine type by considering the asymptotic scaling of return times pointwise  or w.r.t. uniform convergence 
(resp. condition (R) and (U)).  
In the case of circle rotations all the above conditions are equivalent. For interval exchange maps of three intervals
we show  that (D) and (A) are equivalent and imply (Z), (U) and (R) which are equivalent among them.
For maps of four intervals or more we prove several results, the only relation which we cannot decide is whether (Z) implies (R)
or not. 
\end{abstract}

\maketitle

\section{Introduction}

Let $\theta$ be an irrational number: its \emph{type} $\eta\ge 1$ is defined by
$$ \eta = \sup \{ \beta : \liminf_{j \to \infty} j^{\beta} \| j \theta \| = 0 \} $$
where $\|\cdot \|$ denotes the distance from the nearest integer. 
An irrational number is called of \emph{Roth-type} if and only if $\eta =1$. 
This statement is equivalent that 
for every $\varepsilon >0$ there exists a positive constant $C_\varepsilon$
such that 
$$\left| \theta - \frac pq \right| \geq \frac{C_\varepsilon }{q^{2+\varepsilon}} \quad \text{ for all rationals } \ \frac pq.$$

The set of irrational numbers of Roth type has Lebesgue measure 1 and  
contains all algebraic irrational numbers. 
Also it is invariant under the modular group $\textrm{GL}(2,\mathbb Z)$.

Let $(q_n)_{n\in\mathbb{N}}$ be the sequence of the denominators of the principal convergents of the irrational $\theta$ and let $(a_n)_{n\in\mathbb{N}}$ be the sequence of its partial quotients.
Roth type irrationals can also be determined by the growth conditions of $(q_n)_{n\in\mathbb{N}}$: 
for all $\varepsilon >0$ there is a constant  $C_\varepsilon$ such that 
$$q_{n+1} < C_\varepsilon q_n^{1+\varepsilon}.$$
This condition is also equivalent to the growth condition of the partial quotients $(a_n)_{n\in\mathbb{N}}$: 
for all $\varepsilon >0$ there is a constant  $C_\varepsilon$ such that 
$$a_{n+1} < C_\varepsilon q_n^{\varepsilon}.$$

Roth type condition for the irrational $\theta$ can also be given in terms of the dynamics of the associated rotation $R_\theta : x\mapsto x+\theta$ on $\mathbb T=\mathbb R/\mathbb Z$. 
One arises by considering the cohomological equation associated to the rotation $R_\theta$  (see, e.g., \cite{MMY1}). 
Another dynamical characterization of Roth type rotations is obtained by means of the asymptotic scaling laws of first return times and will be recalled below.  
Finally, we consider how evenly an orbit of the rotation is distributed.
If the rotation is of Roth type, then for all $\varepsilon >0$ there is a constant $C_\varepsilon$
such that the minimum distance between points belonging to a finite segment of orbit made $n$  iterates should be bigger than $C_\varepsilon n^{-(1+\varepsilon )} $.

In this paper we investigate the relationship among several not-necessarily equivalent generalizations of the definitions of the Roth type given above to interval exchange maps. 

Let $r>0$ and let  $\tau_r (x)$ be the return time to $r$-neighborhood of $x$
\begin{equation}\label{tau} \tau_r (x) =  \min \{ j \ge 1 : d( T^j x , x ) < r \}\, .  \end{equation}
For an irrational circle rotation (\cite{ChoeSeo}) we have that
\begin{align*}
\varliminf_{r \to 0^+} \frac{ \log \tau_{r} (x) }{ - \log r} &= \frac{1}{\eta},&
\varlimsup_{r \to 0^+} \frac{ \log \tau_{r} (x) }{ - \log r} &= 1.
\end{align*}
Therefore, the rotation number is of Roth type if and only if 
$$
\lim_{r \to 0^+}  \frac{ \log \tau_{r} (x) }{ - \log r}  = 1.
$$

Irrational circle rotations are the prototype of quasiperiodic dynamics and can be generalized as interval exchange maps. 
An interval exchange map $T$ on an interval $I$ is a one-to-one map to itself which is a translation on each finite number of subinterval partition of $I$. 
The map $T$ is an orientation preserving piecewise isometry and preserves the Lebesgue measure. 
Let $d \ge 2$ be the number of the subintervals on which $T$ is a translation.
If $d = 2$, the interval exchange map $T$ corresponds the rotation of circle.

A typical interval exchange map is minimal(\cite{Ke1}).
However, minimality condition for the interval exchange map does not imply unique ergodicity(\cite{Ke2, KeyNew}).
But still almost every interval exchange map is uniquely ergodic(\cite{Ma,Ve1}) and weakly mixing(\cite{AF}).

The modular group $\textrm{GL}(2,\mathbb Z)$ plays an important role for the study of rotation of circle with renormalization scheme associated to the continued fraction algorithm. 
It was generalized by Rauzy and Veech for interval exchange maps by introducing the induced map on appropriated subintervals(\cite{Ra, Ve1}).
The continued fraction algorithm for interval exchange maps is ergodic on the parameter space of interval exchange maps with respect to an absolutely continuous invariant measure with infinite mass. 

Zorich considered an acceleration scheme to produce an ergodic finite invariant measure on the parameter space of the interval exchange maps(\cite{Z1}).
For the rotational case ($d=2$), Zorich's map indeed corresponds the Gauss map which is an acceleration of the Faray map which does not have a probability absolute continuous invariant measure. 

A further acceleration of the Zorich algorithm was studied in \cite{MMY1} by Marmi, Moussa and Yoccoz. They considered a more accelerated algorithm which also preserves an ergodic finite absolute continues invariant measure in the investigation of the regularity of the solutions of the cohomological equation associated to interval exchange maps.
Both the accelerations by Zorich and by Marmi-Moussa-Yoccoz are reduced to the Gauss map for $d=2$.

The notion of Roth type interval exchange map was introduced in \cite{MMY1}: this is a natural extension of Roth type irrational circle rotations and Roth type interval exchange maps form a full measure set in the parameter space of interval exchange maps. 
In \cite{KimMarmi} it was proved that for Roth type interval exchange maps the recurrence time has the same scaling behaviour as for irrational rotations, namely
$$ \lim_{r \to 0^+} \frac{ \log \tau_{r} (x) }{ - \log r} = 1, \text{ a.e. } x.$$

The Roth type condition for the irrational rotation can be generalized to the interval exchange map in several different ways. 
We consider arithmetic characterization using the Roth type growth condition for Marmi-Moussa-Yoccoz cocycle (Condition (A)) and the Roth type growth condition for Zorich cocycle (Condition (Z)).
Uniform return time condition (Condition (U)) and pointwise return time condition (Condition (R)) are defined in terms of 
the dynamics of the map in phase space instead of its evolution in parameter space as is the case for conditions (A) and (Z).
We also consider Roth type condition for the minimal distance between discontinuities (Condition (D)).
In Section~\ref{sec_main_results} these Diophantine conditions for interval exchange maps are given in detail.
In this article we show the relations between the Diophantine conditions 
especially the equivalence of Roth type growth condition for Marmi-Moussa-Yoccoz cocycle and Roth type condition for minimal distance between discontinuities. 
In \cite {MMY1}, it is cited that the relation between them is not clear. (\cite{MMY1}, Sec 1.3.1. Remark 2)


After completing this paper the author noticed the recent work by Marmi, Moussa, and Yoccoz (\cite{MMY2}). 
They also considered the equivalence of Condition (A) and Condition (D) (\cite{MMY2}, Proposition C.1).

\section{Background on continued fraction algorithms for interval exchange maps}\label{sec_background}

An interval exchange map is determined by the combinatorial data of the permutation and the length data of subintervals.
Let $\mathcal A$ be a finite set for the name of subintervals.
We denote the combinatorial data by two bijections $(\pi_t, \pi_b)$
from $\mathcal A$ onto $\{ 1, 2, \dots, d\}$.
which indicate the order of the subintervals before and after the interval exchange map.
The length data, denoted by $(\lambda_\alpha)_{\alpha\in \mathcal A}$, give the length of the corresponding subintervals. 

We set 
$$\lambda^* := \sum_{\alpha \in \mathcal A} \lambda_A, \quad I := [ 0, \lambda^*)$$
and 
$$ p_\alpha := \sum_{\pi_t(\beta) < \pi_t(\alpha)} \lambda_\beta,
\quad I_\alpha := [p_\alpha, p_\alpha+\lambda_\alpha ), \quad I=\bigsqcup_{\alpha \in \mathcal A} I_\alpha. $$
Then the interval exchange map $T$ associated to the combinatorial data $(\pi_t, \pi_b)$ and the length data $(\lambda_\alpha)_{\alpha \in \mathcal A}$ is a bijective map on $I$ given by
$$ T(x) = x + \sum_{\pi_b(\beta) < \pi_b (\alpha)} \lambda_\beta - \sum_{\pi_t(\beta) < \pi_t (\alpha)} \lambda_\beta 
\ \text{ for } \ x \in I_\alpha.$$
Note that $T$ is discontinuous at $p_\alpha$ with $\pi_t(\alpha) >1$.

We will consider only combinatorial data $(\pi_t, \pi_b )$ which are \emph{admissible}, 
in the sense that for all $k=1,2,\dots, d-1$, we have
$$ \pi_t^{-1}(\{1,\ldots ,k\}) \ne \pi_b^{-1}(\{1,\ldots ,k\}). $$

An interval exchange map $T$ is said to have the \emph{Keane property}
if there exist no $\alpha,\beta \in \mathcal A$ and positive integer $m$ such that $T^m (p_\alpha)= p_\beta$ and $\pi_t (\beta )>1$.
An admissible interval exchange map with rationally independent length data has the Keane property and an interval exchange map with Keane's property is minimal\cite{Ke1}.
Thus Keane's property corresponds to the notion of irrationality for interval exchange maps.

For admissible interval exchange maps with the Keane property we can
introduce the generalization of continued fraction algorithm to interval exchange maps due to the work of Rauzy \cite{Ra}, Veech \cite{Ve1} and Zorich \cite{Z1,Z2}.
We refer to \cite{MMY1, Y1, Y2} and references therein for the detailed discussions and proofs.

Let $(\pi_t, \pi_b)$ be an admissible pair. We define two
new admissible pairs $\mathcal{R}_t (\pi_t, \pi_b)$ and $\mathcal{R}_b (\pi_t, \pi_b)$ as follows:
let $\alpha_t$ and $\alpha_b$ be the (distinct) elements of $\mathcal{A}$ such that
$\pi_t(\alpha_t)=\pi_b(\alpha_b)=d$; one has
\begin{align*}
\mathcal{R}_t (\pi_t, \pi_b) &= (\pi_t, \hat{\pi}_b), \\
\mathcal{R}_b (\pi_t, \pi_b) &= (\hat{\pi}_t, \pi_b),
\end{align*}
where
\begin{align*}
\hat{\pi}_b (\alpha) &= \begin{cases}
\pi_b(\alpha ) & \textrm{if } \pi_b(\alpha )\le \pi_b(\alpha_t),\\
\pi_b(\alpha )+1 & \textrm{if } \pi_b(\alpha_t )<\pi_b(\alpha)<d,\\
\pi_b(\alpha_t )+1 & \textrm{if } \alpha=\alpha_b$, ($\pi_b(\alpha_b)=d);\\
\end{cases}
\\
\hat{\pi}_t (\alpha) &= \begin{cases}
\pi_t(\alpha ) & \textrm{if } \pi_t(\alpha )\le \pi_t(\alpha_b),\\
\pi_t(\alpha )+1 & \textrm{if } \pi_t(\alpha_b )<\pi_t(\alpha)<d,\\
\pi_t(\alpha_b )+1 & \textrm{if } \alpha=\alpha_t$, ($\pi_t(\alpha_t)=d).
\end{cases}
\end{align*}

\noindent
The {\it Rauzy class} of $(\pi_t, \pi_b)$ is the
set of admissible
pairs obtained by saturation of $(\pi_t, \pi_b)$ under the action of
$\mathcal{R}_t$ and $\mathcal{R}_b$.
The {\it Rauzy diagram} has for vertices the elements
of the Rauzy class, each vertex $(\pi_t,\pi_b)$
being the origin of two arrows joining $(\pi_t,\pi_b)$ to $\mathcal{R}_t (\pi_t, \pi_b)$, $\mathcal{R}_b (\pi_t, \pi_b)$. 
See Figure~\ref{diagram3} and \ref{diagram4} for the Rauzy diagrams for a 3-interval map and a 4-interval map.
For an arrow joining $(\pi_t,\pi_b)$ to $\mathcal{R}_t (\pi_t, \pi_b)$ (respectively $\mathcal{R}_b (\pi_t, \pi_b)$) the element $\alpha_t \in \mathcal{A}$ (respectively $\alpha_b \in \mathcal{A}$) is called the {\it winner} and 
the element $\alpha_b \in \mathcal{A}$ (respectively $\alpha_t \in \mathcal{A}$) is called the {\it loser}.

We say that $T$ is of {\it top type} (respectively {\it bottom type})
if one has $\lambda_{\alpha_t}>\lambda_{\alpha_b}$ (respectively $\lambda_{\alpha_b}>\lambda_{\alpha_t}$);
we then define a new interval exchange map $\mathcal{V} (T)$ by the following data:
the admissible pair $\mathcal{R}_t(\pi_t, \pi_b)$ and the lengths
$(\hat{\lambda}_\alpha)_{\alpha\in\mathcal{A}}$ given by
$$\begin{cases}
\hat{\lambda}_\alpha = \lambda_\alpha & \textrm{if } \alpha\not=\alpha_t,\\
\hat{\lambda}_{\alpha_t} = \lambda_{\alpha_t}-\lambda_{\alpha_b} & \textrm{otherwise}
\end{cases} 
$$
for the top type $T$; 
the admissible pair $\mathcal{R}_b(\pi_t, \pi_b)$ and the lengths
$$\begin{cases}
\hat{\lambda}_\alpha = \lambda_\alpha & \textrm{if } \alpha\not=\alpha_b,\\
\hat{\lambda}_{\alpha_b} = \lambda_{\alpha_b}-\lambda_{\alpha_t} & \textrm{otherwise}
\end{cases} 
$$
for the bottom type $T$.

The interval exchange map $\mathcal{V} (T)$ is the first return map of  $T$ on
$\left[0, \sum_{\alpha} \hat \lambda_\alpha\right)$.
We also associate to $T$ the arrow in the Rauzy diagram
joining $(\pi_t,\pi_b)$ to $\mathcal{R}_t(\pi_t, \pi_b)$ or $\mathcal{R}_b(\pi_t, \pi_b)$.
Iterating this process, we obtain a sequence of interval exchange map
$T(n) = \mathcal{V}^n (T)$, $n\ge 0$ and an infinite path in the Rauzy diagram starting from $(\pi_t,\pi_b)$. 
In fact a further
property of irrational interval exchange maps (i.e.\ with the Keane property)
is that every letter in $\mathcal A$ is taken as a winner infinitely many times
in the infinite path (in the Rauzy diagram) associated to $T$.
This property is fundamental in order to be able
to group together several iterations of
$\mathcal{V}$ to obtain the accelerated Zorich continued fraction algorithm
introduced in \cite{MMY1}.

For an arrow $\gamma$ with winner $\alpha$ and loser $\beta$ in the Rauzy diagram,
let
$$ B_\gamma = \mathbb I + E_{\beta\alpha}$$
where $\mathbb I$ is the identity matrix and $E_{\beta\alpha}$ is the elementary matrix with the only nonzero element at $(\beta,\alpha)$ which is equal to 1.
For a finite path $\underline \gamma = ( \gamma_1, \dots, \gamma_n)$ in the Rauzy diagram 
we have a $\textrm{SL} (\mathbb{Z}^{\mathcal A})$ matrix with nonnegative entries
$$ B_{\underline \gamma} = B_{\gamma_n} \cdots B_{\gamma_1}.$$
Let $\gamma^T(m,n) = \gamma (m,n)$ be the path in the Rauzy diagram from $\pi(m)$ to $\pi(n)$ for $m\le n$
and denote by $$Q(m,n)= B_{\gamma (m,n)} \text{ and } Q(n) = Q (0,n).$$
Let $\lambda(n)$ be the length data of $T(n)$.
Then we have
\begin{equation}\label{lambdaQ}
\lambda(m) = \lambda(n) Q(m,n).
\end{equation}

For $m\le n$, $T(n)$ is the induced map of $T(m)$
on $I(n)=[0, \lambda^*(n) )$,
where $\lambda^*(n) = \sum_{\alpha \in \mathcal A} \lambda_\alpha (n)$;
the return time on $I_\beta (n)$ to $I(n)$ under the iteration $T(m)$ is
$Q_{\beta}(m,n) :=  \sum_\alpha Q_{\beta\alpha}(m,n)$
and the time spent in $I_\alpha (m)$ is $Q_{\beta\alpha}(m,n)$.
By (\ref{lambdaQ}) we have 
\begin{equation}\label{QQ}
\lambda^* = \sum_{\alpha\beta} \lambda_\beta(n) Q_{\beta\alpha}(n)
= \sum_{\beta} \lambda_\beta(n) Q_{\beta}(n).
\end{equation}
Moreover, we have 
\begin{equation}\label{partition}
 [0,\lambda^*) =  \bigsqcup_{\alpha \in \mathcal A} \left( \bigsqcup_{i = 0}^{Q_\alpha (n)-1} T^i ( I_\alpha(n) ) \right) .
\end{equation}

Zorich's accelerated continued fraction algorithm is obtained by considering 
$(\mathcal{V}^{n_k})_{k\ge 0}$ where $({n_k})_{k\ge 0}$ is the following sequence: 
$n_0=0$ and $n_{k+1}>n_k$ is chosen so as to assure that 
$\gamma(n_{k},n_{k+1})$ is the longest path  whose arrows have the same winner.

The further acceleration algorithm by Marmi-Moussa-Yoccoz, which was introduced in \cite{MMY1}, is obtained by considering 
$(\mathcal{V}^{m_k})_{k\ge 0}$ where $({m_k})_{k\ge 0}$ is defined as follows: 
$m_0 = 0$ and 
$m_{k+1}>m_k$ is the largest integer such that not all letters in $\mathcal A$ 
are taken as winner by arrows in $\gamma(m_{k},m_{k+1})$.

Let\footnote{We warn the reader that our notations are slightly different from the one followed in \cite{Y2}: in this
paper the matrices $A(k)$ denote the matrices obtained by the accelerated Zorich algorithm introduced in \cite{MMY1} and $Z(k)$ are those obtained by the original Zorich algorithm, whereas in \cite{Y2} the former were denoted $Z(n)$ since 
the latter were never used  explicitely.}
\begin{align*}
&\hbox{Zorich cocycle}& Z (k) & = Q(0,n_k), &Z (k,\ell) & = Q(n_k,n_\ell),\\
&\hbox{Marmi-Moussa-Yoccoz cocycle}&  A (k) & = Q(0,m_k), &A (k,\ell) &= Q(m_k,m_\ell).
\end{align*}

The most important virtue of the Marmi-Moussa-Yoccoz cocycle is the following:

\begin{lemma}[\cite{MMY1} Lemma 1.2.4]\label{MMYlemma}
Let $r \ge \max ( 2d - 3, 2)$. Then we have 
$$ A_{\beta \alpha}(k,k + r) >0 \text{ for all } \alpha ,\beta\in\mathcal{A}. $$
\end{lemma}

The following inequality follows easily from (\ref{QQ})
\begin{equation}\label{QQQ}
\min_{\alpha \in \mathcal A} \lambda_\alpha (n) \le \frac{\lambda^*}{\| Q(n) \|} \le \max_{\alpha \in \mathcal A} \lambda_\alpha (n),
\end{equation}
where the norm of a matrix $ B $ is simply the sum of the absolute values of its entries.
This is the 
norm that we will use for matrices throughout the whole paper. 
We assume that $\lambda^* = 1$ unless it is specified.

\section{Diophantine conditions for interval exchange maps}
\label{sec_main_results}

If one considers the dynamics in parameter space of interval exchange maps one can introduce three slightly different Diophantine conditions:

\begin{itemize}
\item[(A)] Roth type growth condition for the Marmi-Moussa-Yoccoz cocycle : \\
For any $\varepsilon >0 $ there exist $C_\varepsilon >0$ such that for all $k \ge 1 $ we have 
$$ \| A(k,k+1) \| \le C_\varepsilon \| A(k) \|^\varepsilon . $$

\item[(Z)] Roth type growth condition for the Zorich cocycle : \\
For any $\varepsilon >0 $ there exist $C_\varepsilon >0$ such that for all $k \ge 1 $ we have 
$$ \|  Z(k,k+1) \| \le C_\varepsilon \| Z (k) \|^\varepsilon . $$
\end{itemize}

Let $\Delta(T)$ be the minimum distance between the discontinuity points of $T$ or the end points 0 and 1.

\begin{itemize}
\item[(D)] Roth type condition for the minimal distance between discontinuities : \\
For any $\varepsilon >0 $ there exist $C_\varepsilon >0$ such that for all $n \ge 1 $ we have
$$ \Delta (T^n) \ge \frac{C_\varepsilon}{n^{1+\varepsilon}} . $$
\end{itemize}

If one considers the dynamics of an interval exchange map in phase space then one can introduce
two slightly different Diophantine conditions: 
\begin{itemize}
\item[(R)] Pointwise return time condition : 
$$ \lim_{r \to 0} \frac{ \log \tau_r (x) }{-\log r} = 1 \text{ for almost every } x.$$

\item[(U)] Uniform return time condition : 
$$ \lim_{r \to 0} \frac{ \log \tau_r (x) }{-\log r} = 1 \text{ uniformly}.$$
\end{itemize}
Here $\tau_r(x)$ be the first return time to $r$-neighborhood of $x$ defined in (\ref{tau}).

Here and in what follows the matrix norm denoted by $\|Q\| =\sum_{\alpha \beta} |Q_{\alpha \beta}|$.
In the case of circle rotations (interval exchange map with $d=2$) the three conditions in parameter space (namely (A), (Z) and (D)) are equivalent\footnote{For an irrational rotation, 
$ Z(1) =  A(1) = \begin{pmatrix} 1 & 0  \\ a_1 -1 & 1 \end{pmatrix}$,
$Z(k-1,k) = A(k-1,k) = \begin{pmatrix} 1 & 0 \\ a_k & 1 \end{pmatrix}$
or $\begin{pmatrix} 1 &  a_k \\ 0   & 1 \end{pmatrix}$ 
and $ Z(k) = A(k) = \begin{pmatrix} q_{k-1} - p_{k-1} &  p_{k-1} \\ q_k - p_k  & p_k \end{pmatrix}$ or  
$\begin{pmatrix} q_k - p_k & p_k \\ q_{k-1} - p_{k-1}  & p_{k-1} \end{pmatrix}$ depending on $k$ is odd or even.
Therefore, we have
$\|Z(k,k+1)\| =\| A(k,k+1) \| = a_{k+1} + 2$, $\| Z (k) \| = \| A(k) \| = q_k + q_{k-1} $
and   Condition (A) and (D) are equivalent to the statement that for any $\varepsilon >0$ there is a positive constant $C_\varepsilon$ such that
$ a_{k+1} \le  C_\varepsilon  q_k^\varepsilon, $
which just the Roth type condition for the irrational rotation number.}, as well as the two conditions in phase space ((R) and (U)). 
In \cite{ChoeSeo} the equivalence for circle rotations between the two sets of conditions
(Roth type in parameter space and the return time characterization) was proved. 
For general interval exchange maps  in \cite{KimMarmi} it is proved that (A) implies (R).

In this article, we investigate the relation among Condition (A), (Z), (D), (U) and (R) for general interval exchange maps (with $d\ge 3$). It is not difficult to verify that from the definitions one has 

(A) $\Rightarrow$ (Z) and 

(U) $\Rightarrow$ (R).

In Section~\ref{3iem_sec} we will prove that for 3-interval exchange maps (Z), (U)  and (R) are equivalent 
and (A) and (D) are equivalent. Moreover, in the same Section, we construct a family of 3-interval exchange maps  which all satisfy Condition (U) but neither Condition (A) nor (D).
For general maps with $d\ge 4$ we will establish that 

(A) $\Leftrightarrow$ (D) : this is proved in Section~\ref{AandD} 

(D) $\Rightarrow$ (U) : this is proved in  Section~\ref{DtoU}

(U) $\Rightarrow$ (Z) : this is proved in  Section~\ref{UtoZ}

(R) does not imply (Z) : this is proved in  Section~\ref{exm_sec1}

(Z) does not imply (U) : this is proved in  Section~\ref{exm_sec2}

The only relation we could not decide is whether (Z) implies (R) or not.

\section{Condition (A) is equivalent to Condition (D)}\label{AandD}

For each $\alpha \in \mathcal A$ let
$$ p_\alpha (n) = \sum_{\pi_t^{(n)}(\beta) < \pi_t^{(n)}(\alpha)} \lambda_\beta (n), \qquad 
q_\alpha (n) = \sum_{\pi_b^{(n)}(\beta) < \pi_b^{(n)}(\alpha)} \lambda_\beta (n), $$
and
$$ I_\alpha (n) =  [p_\alpha (n), p_\alpha (n) + \lambda_\alpha (n)).$$
Then
$$T(n) ( I_\alpha (n) ) = [ q_\alpha (n), q_\alpha(n) + \lambda_\alpha(n) ).$$
Denote by $D(T)$ the set of discontinuity points of $T$.
Let 
$$\mathcal A' = \{ \alpha \in \mathcal A : \pi_t (\alpha) > 1 \}. $$
Note that we have $D(T(n)) = \{ p_\alpha (n) : \alpha \in \mathcal A'\}$.

\begin{lemma}\label{lemmalr}
For each $\alpha \in \mathcal A'$, then we have 
$$T^i ( p_\alpha (n) ) \in D(T)$$
for some $i$ such that $0 \le i < Q_\alpha (n) $.
Conversely,  if $p \in D(T)$, then 
$$p = T^i ( p_\alpha(n) )   \text{ for an } \alpha \in \mathcal A' \text{ and } 0 \le i < Q_\alpha (n). $$
\end{lemma}

\begin{proof}
We will prove the statement by induction. 
Both statements are trivial if $n=0$.
Assume that the lemma holds for $n>1$.
Let $\alpha$ and $\beta$ be such that
$$\pi_t^{(n)} (\alpha) = d, \quad \pi_b^{(n)} (\beta) = d.$$

If $T$ is of bottom type, i.e., $\beta$ is the winner,
then $ \lambda_\beta (n) > \lambda_\alpha (n)$.
Then $$ T^{Q_\beta (n)} (p_\alpha (n+1) ) = p_\alpha (n).$$
Since  $$ Q_\alpha (n+1) = Q_\alpha (n) + Q_\beta (n),$$
the lemma holds for $n+1$.

If $T$ is of top type, then $p_\alpha (n+1) = p_\alpha (n)$ for all $\alpha \in \mathcal A$,
so the lemma holds for $n+1$.
\end{proof}

\begin{lemma}\label{lemmalr2}
If $0 < n \le \min_{\alpha \in \mathcal A} A_\alpha (k)$, then we have
\begin{equation*}
\begin{split} 
D (T^n) &\subset \bigcup_{\alpha \in \mathcal A'} \{ T^i (p_\alpha (m_k)) : - A_\beta (k) \le i < A_\alpha (k) \text{ where } p_\alpha(m_k) \in T(m_k) ( I_\beta (m_k) ) \} \\
&= \bigcup_{p \in D(T(m_k)^2)} \{ T^i ( p ) : 0  \le i < A_\alpha (k), \text{ where } p \in I_\alpha (m_k) ) \}.
\end{split}
\end{equation*}
\end{lemma}

\begin{proof}
By Lemma~\ref{lemmalr}, if $p \in D(T)$, then 
we have $p  = T^i (p_\alpha (m_k) ) $ for some $\alpha \in \mathcal A'$, $p_\alpha (m_k) >0$ and $i$ such that $0 \le i <  A_\alpha (k)  $.

Since $ D(T^n) = D(T) \cup T^{-1} (D(T)) \cup \dots \cup T^{-(n-1)} (D(T))$,
if $p \in D(T^n)$, then we have $p  = T^i (p_\alpha (m_k) ) $ for some $\alpha$ and $i$ such that $ -n +1 \le i < A_\alpha (k) $.
From the assumption
$ n \le \min_\beta A_\beta (k)$ we have the inclusion and
since the discontinuity point of $T(m_k)^2$ is either the discontinuity point of $T(m_k)$ or the preimage of them,
we complete the proof.
\end{proof}

\begin{lemma}\label{lemmalr3}
If $\ell > k$ satisfies $$ \lambda^* (m_{\ell})< \lambda_\sigma (m_k), \text{ where } \pi_t^{(m_k)} (\sigma) = 1,$$  
then we have
$$ \min_{\alpha \in \mathcal A} \lambda_\alpha (m_{\ell}) \le \Delta \left( T(m_k)^2 \right) .$$
\end{lemma}

\begin{proof}
Choose $p$ be a discontinuity point of $T(m_k)$.
Then $p = T(m_k)^i (q)$, $0 \le i < A_\alpha(k,{\ell})$ for some $q = p_\alpha (m_\ell) \in D \left( T(m_\ell) \right)$.  
From the the assumption $ \lambda^* (m_\ell) < \lambda_\alpha (m_k)$ we have $p \ne q$ and $1 \le i < A_\alpha(k, \ell)$. 
Therefore, if $p \in D \left( T(m_k)^2 \right) = D \left( T(m_k) \right) \cup T(m_k)^{-1} \left( D \left( T(m_k) \right) \right)$, then 
$p =  T(m_k)^i (q)$ for some $q \in D \left( T(m_\ell) \right)$ with $0 \le i < A_\alpha(k, \ell) $.
Since the minimum distance among $T(m_k)^i (q)$, $0 \le i < A_\alpha(k, \ell) $ is $\Delta \left( T(m_k)^2 \right)$, which completes the proof.
\end{proof}

\begin{lemma}[\cite{MMY1}, p.835]\label{22}
If $T$ satisfies Condition (A), then
$$
\max_{\alpha \in \mathcal A} \lambda_\alpha (m_k) \le C_\varepsilon \min_{\alpha \in \mathcal A} \lambda_\alpha (m_k) \cdot \| A(k)\|^\varepsilon.
$$
\end{lemma}
Combining with (\ref{QQQ}), we have
\begin{equation}\label{222}
\frac{1}{\| A(k) \|} \le \max_{\alpha \in \mathcal A} \lambda_\alpha (m_k)
\le C_\varepsilon \min_{\alpha \in \mathcal A} \lambda_\alpha (m_k) \cdot \| A(k)\|^\varepsilon.
\end{equation}

\begin{theorem}
If $T$ satisfies condition (A), then it also satisfies Condition (D).
\end{theorem}

\begin{proof}
For each positive integer $n$ we have $k$ such that 
\begin{equation}\label{nak} \min_{\alpha \in \mathcal A} A_\alpha (k -1) < n \le \min_{\alpha \in \mathcal A} A_\alpha (k).\end{equation}
Then by Lemma~\ref{lemmalr2} and (\ref{partition}) we have
\begin{equation}\label{44} 
\Delta (T^{n}) \ge \Delta \left( T(m_k)^2 \right).\end{equation}
By Lemma~\ref{MMYlemma} 
there is a constant $r = \max(2d-3,2)$ such that 
$$ A_{\alpha\beta}(k, k+r) > 0  \text{ for all } \alpha, \beta \in \mathcal A,$$
which implies that
$$ \lambda^* (m_{k+r}) < \min_{\alpha \mathcal A} \lambda_\alpha (m_k).$$  
Therefore, by Lemma~\ref{lemmalr3} and (\ref{44}), we have
$$ \Delta (T^{n}) \ge \Delta \left( T(m_k)^2 \right) 
\ge \min_{\alpha \in \mathcal A} \lambda_\alpha (m_{k+r}). $$

By the definition of Condtion (A) for any $\varepsilon >0$ we can choose a constant $C_\varepsilon$ such that
\begin{equation}\label{331}
\| A(k+r+1) \| <  C_\varepsilon \| A(k) \|^{1+\varepsilon}. \end{equation}
By Lemma~\ref{MMYlemma} we have 
\begin{equation}\label{333}
\begin{split}
\min_{\alpha \in \mathcal A} A_\alpha (k+r) &= \min_{\alpha \in \mathcal A} \left( \sum_{\beta \in \mathcal A} A_{\alpha \beta} (k, k+r) A_\beta (k) \right) \\
&> \max_{\beta \in \mathcal A} A_\beta (k) \ge \frac 1d \| A(k) \| > \frac{C_\varepsilon}{d} \| A(k+r+1) \|^{1/(1+\varepsilon)}.
\end{split}
\end{equation}
Hence, we have for some constants $C'_\varepsilon$ and $C'_\varepsilon$
\begin{align*}
\Delta (T^{n}) &\ge \min_{\alpha \in \mathcal A} \lambda_\alpha (m_{k+r})
\ge C_\varepsilon \| A(k+r) \|^{-(1+\varepsilon)} &\text{by (\ref{222}),} \\
&\ge C'_\varepsilon \| A(k) \|^{-(1+\varepsilon)^2} &\text{by (\ref{331}),} \\
&> C''_\varepsilon \left( \min_{\alpha \in \mathcal A} A_\alpha (k-1) \right)^{-(1+\varepsilon)^3}
&\text{by (\ref{333})}, \\ 
&> C''_\varepsilon n^{-(1+\varepsilon)^3} &\text{by (\ref{nak})}.
\end{align*}
\end{proof}

Now we prove the other direction.

We have
$$\sum_{\alpha , \beta \in \mathcal A} \lambda_\alpha (m_{k+1}) A_{\alpha \beta} (k,k+1) = \lambda^* (m_k)$$
so
$$ \min_\alpha \lambda_\alpha (m_{k+1}) \cdot \| A(k,k+1) \| <  \lambda^* (m_k)
< \max_\alpha \lambda_\alpha (m_{k+1}) \cdot \| A(k,k+1) \| . $$

\begin{lemma}\label{nnnlemma}
Suppose that $T$ does not satisfy Condition (A). 
Then for some $r > 0$ there are infinitely many $k$ such that
$$ \min_{\alpha \in \mathcal A} \lambda_{\alpha}(m_k) < \lambda^* (m_k)^{1 + r}.$$
\end{lemma}

\begin{proof}
For each $k$ let $\alpha(k) \in \mathcal A$, depending on $k$, be the letter which is not taken as the winner of the arrows in the path $\gamma(k,k+1)$.
Then
$$ \lambda_\alpha (m_k) = \lambda_\alpha (m_{k+1}) .$$

Let $\varepsilon(k)$ be given by
$$\|A(k,k+1)\| = \|A(k) \|^{\varepsilon(k)}. $$
Now we have two cases:

Case (i) :
$\lambda_\alpha (m_k) \cdot \sqrt{\| A(k,k+1) \|} <  \lambda^* (m_k) $  \\
We have
\begin{equation}\label{casei}
\frac{ \lambda_\alpha (m_k) }{  \lambda^* (m_k)}  < \frac{1}{ \sqrt{\|A(k,k+1) \|} } = \frac{1}{  \|A(k) \|^{\varepsilon(k) /2} } <  \lambda^* (m_k)^{\varepsilon(k) /2}.
\end{equation}
The last inequality follows from (\ref{QQQ}). 
 
Case (ii) :
$\lambda_\alpha (m_k) \cdot \sqrt{\|A(k,k+1) \|} \ge  \lambda^* (m_k) $ \\
Since
$$\sum_{\alpha , \beta \in \mathcal A} \lambda_\alpha (m_{k+1}) A_{\alpha \beta} (k,k+1) = \lambda^* (m_k),$$
we have
$$ \min_\alpha \lambda_\alpha (m_{k+1}) \cdot \| A(k,k+1) \| <  \lambda^* (m_k) < \max_\alpha \lambda_\alpha (m_{k+1}) \cdot \| A(k,k+1) \| . $$
Thus, there is $\beta \in \mathcal A$ such that 
$$ \lambda_\beta (m_{k+1})  \cdot \|A(k,k+1) \| <  \lambda^* (m_k) \le
\lambda_\alpha (m_k) \cdot \sqrt{\|A(k,k+1) \|} .$$
Therefore, we have
$$ \lambda^* (m_{k+1}) > \lambda_\alpha (m_{k+1}) = \lambda_\alpha (m_k) > \lambda_\beta (m_{k+1})  \cdot \sqrt{\| A(k,k+1) \|}  $$
and
$$ \frac{\lambda_\beta (m_{k+1})}{ \lambda^* (m_{k+1})}  < \frac{1}{\sqrt{\| A (k,k+1) \|} }  = \frac{1}{  \|A(k) \|^{\varepsilon (k)/2} }.$$

Since $ \|A(k) \|^{1+\varepsilon(k)} = \|A(k,k+1)\| \cdot \|A(k) \| \ge \|A(k,k+1) A(k)\|  = \|A(k+1) \| , $  
we have
\begin{equation}\label{caseii}
\frac{\lambda_\beta (m_{k+1})}{ \lambda^* (m_{k+1})}  < \frac{1}{ \|A(k+1) \|^{\varepsilon /2 (1+\varepsilon)}} < \lambda^* (m_{k+1})^{\varepsilon /2 (1+\varepsilon)},
\end{equation}
where the last inequality is from (\ref{QQQ}).

Suppose that $T$ does not satisfy Condition (A).
Then $\limsup_{k} \varepsilon(k) > 0$.
The lemma then follows by applying inequalities 
(\ref{casei}) and (\ref{caseii}).
\end{proof}

\begin{lemma}\label{newlemma}
Let $\alpha \in \mathcal A$ be the winner of $\gamma(n-1,n)$ and the loser of $\gamma(n,n+1)$.
If $\lambda_\alpha (n) <  \lambda^* (n)^{1 + r}$, $r >0$ for large $n$, 
then there is an integer $s$, $1 \le s < d$, such that
$$ \Delta \left(T^{\lfloor 2/  \lambda^* (n)^{1+sr/d}\rfloor} \right) < (d-1)  \lambda^* (n)^{1+(s+1)r/d}. $$
\end{lemma}

\begin{proof}
Assume that $ \lambda^* (n)$ is small enough that $ \lambda^* (n)^{{r}/{d}} < 1/d$.

Let for $0 \le i < d $
$$
\mathcal A_i = \{ \beta \in \mathcal A : \lambda^* (n)^{1 + {(i+1)r}/{d}} \le \lambda_\beta (n) < \lambda^* (n)^{1 + {ir}/{d}} \}
$$
and
$$
\mathcal A_d = \{ \beta \in \mathcal A : \lambda_\beta (n) < \lambda^* (n)^{1 + r} \}.
$$
Then, by the assumption, $\alpha \in \mathcal A_d \ne \emptyset$.
Since there is an $\beta \in \mathcal A$ such that 
$\lambda_\beta (n) > \lambda^* (n)/{d} >  \lambda^* (n)^{1+{r}/{d}}$, neither $\mathcal A_0$ is an empty set.

Since there are $d$ elements in $\mathcal A$, there exist an $s$, $1 \le s < d$, such that $\mathcal A_s $ is empty.
Let 
$$
\mathcal A_{\textrm{big}} = \bigcup_{i=0}^{s-1} \mathcal A_i, \qquad
\mathcal A_{\textrm{small}} = \bigcup_{i=s+1}^{d} \mathcal A_i. 
$$
Both of $\mathcal A_{\textrm{big}}$ and $\mathcal A_{\textrm{small}}$ are nonempty.

Take $m$, $m < n$ be the smallest integer as 
no loser in $\gamma(m+1,n)$ belongs to $\mathcal A_{\textrm{big}}$. 
Put $\mu \in \mathcal A_{\textrm{big}}$ as the loser of the arrow $\gamma(m,m+1)$. 
Let $\nu$ be the winner of the arrow $\gamma(m,m+1)$.
Then $\nu \in \mathcal A_{\textrm{small}}$.
(if $\nu \in \mathcal A_{\textrm{big}}$, then $\nu \ne \alpha$ and $\nu$ should be a loser in $\gamma(m+1,n)$)

Hence we have
$\lambda_{\nu} (m+1) =  \lambda_{\nu} (m) - \lambda_{\mu} (m)$ and
\begin{align*}
Q_{\nu} (m+1) &= Q_{\nu} (m) < \frac{1}{\lambda_\nu (m)} < \frac{1}{\lambda_{\mu}(m)} \le \frac{1}{ \lambda^* (n)^{1+{sr}/{d}} }, \\
Q_{\mu} (m+1) &= Q_{\nu} (m) + Q_{\mu} (m)< \frac{1}{\lambda_\nu (m)} + \frac{1}{\lambda_{\mu}(m)} < \frac{2}{\lambda_{\mu}(m)} < \frac{2}{ \lambda^* (n)^{1+{sr}/{d}} }.
\end{align*}

There are two cases:

(i) $\pi_t^{(m)}(\mu ) = d$ and $\pi_b^{(m)}(\nu) = d$: \\
Then we have $\pi_t^{(m+1)}(\nu ) = \pi_t^{(m)}(\nu ) < d$ and 
$\pi_t^{(m+1)}(\mu ) = \pi_t^{(m)}(\nu ) +1$.

Since no letter in $\mathcal A_{\textrm{big}}$ is taken as the winner or the loser of the arrows of $\gamma(m+1,n)$, 
$$ I_\nu (m+1), \ I_\mu (m+1) \subset [0, \lambda^*(n))$$
and
$$ I_\nu (m+1) = \left[ p_\nu (m+1) , p_\mu (m+1) \right ).$$
Since $p_\nu (m+1)$ and $p_\mu (m+1)$ are discontinuity points of $T(n)$ and $\pi_b^{(m+1)}(\nu) = d$, $m+1 \le n$,
we have $$I_\nu (m+1) = \bigsqcup_{\beta \in \mathcal A'} I_\beta (n)  
\ \text{ for some } \mathcal A' \subset A_{\textrm{small}}.$$
Therefore, we have 
\begin{align*}
p_\mu (m+1)-p_\nu(m+1) &=\lambda_\nu(m+1) \\
&< | \mathcal A_{\textrm{small}} | \cdot \lambda^*(n)^{1+{(s+1)r}/{d}} 
\le (d-1)\lambda^*(n)^{1+{(s+1)r}/{d}}.
\end{align*}
Since 
\begin{align*}
p_\mu (m+1) &\in D \left( T^{Q_\mu(m+1)} \right) = D \left( T^{Q_\nu(m) + Q_\mu(m)} \right). \\
p_\nu (m+1) &\in D \left( T^{Q_\nu(m+1)} \right) = D \left( T^{Q_\nu(m) } \right),
\end{align*}
we have
$$ p_\mu (m+1) - p_\nu (m+1) \ge \Delta \left( T^{Q_\mu(m+1)} \right) .$$

(ii) $\pi_t^{(m)}(\nu ) = d$ and $\pi_b^{(m)}(\mu ) = d$: \\
Then we have $\pi_b^{(m+1)}(\nu ) = \pi_b^{(m)}(\nu ) < d$ and 
$\pi_b^{(m+1)}(\mu ) = \pi_b^{(m)}(\nu ) +1$.
Similarly with case (i), we have
$$q_\mu (m+1) - q_\nu (m+1) = \lambda_\nu (m+1) < (d-1) \lambda^*(n)^{1+{(s+1)r}/{d}}$$
Since 
\begin{align*}
q_\mu (m+1) &\in D \left( T^{-Q_\mu(m+1)} \right) = D \left( T^{-Q_\nu(m) - Q_\mu(m)} \right) ,\\
q_\nu (m+1) &\in D \left( T^{-Q_\nu(m+1)} \right) = D \left( T^{-Q_\nu(m) } \right) ,
\end{align*} we have
$$ q_\mu (m+1) - q_\nu (m+1) \ge \Delta \left ( T^{- Q_\mu(m+1)} \right )
= \Delta \left ( T^{Q_\mu(m+1)} \right ) .$$
Note that $\Delta(T) = \Delta(T^{-1})$.
\end{proof}

Now we have the following theorem for the opposite direction.

\begin{theorem}
If $T$ does not satisfy Condition (A), then $T$ does not satisfy Condition (D), neither.
\end{theorem}

\begin{proof}
By Lemma~\ref{nnnlemma} we have $r > 0$ and infinitely many $k$ and $\alpha$ (depending on $k$) satisfying 
$$ \lambda_{\alpha}(m_k)= \min_{\beta \in \mathcal A} \lambda_{\beta}(m_k) < \lambda^* (m_k)^{1 + r}.$$

Let $\ell_k(\alpha) = \max \{n \le m_k: \alpha \text{ is the winner of } \gamma(n-1,n)\}$. 
Then $\alpha$ is the loser of $\gamma ( \ell_k(\alpha), \ell_k(\alpha)+1)$  
or $\alpha$ is the winner of $\gamma ( m_k, m_k+1)$, $m_k = \ell_k(\alpha)$.

By the definition of the Marmi-Moussa-Yoccoz acceleration sequence $m_k$, the winner of $\gamma(m_k-1,m_k)$ and the winner of $\gamma(m_k,m_k +1)$ are different.
Hence, if we put $n = \ell_k(\alpha)$, then $\alpha$ is the winner of $\gamma (n-1, n)$ and the loser of $\gamma (n, n+1)$  
and 
$$ \lambda_{\alpha}(n) = \lambda_{\alpha}(m_k) < \lambda^* (m_k)^{1 + r} \le \lambda^* (n)^{1 + r}.$$

Since $ \lambda_{\alpha}(m_k)= \min_{\beta \in \mathcal A} \lambda_{\beta}(m_k)$,
$\alpha$ cannot be the winner of $\gamma(m_k,m_k +1)$. 
Thus $\alpha$ should be the winner of an arrow in $\gamma(m_{k-1},m_k)$, which yields
$$ m_{k-1} \le \ell_k(\alpha) = n \le m_k.$$ 
Hence, we can choose infinitely many $n$'s satisfying the condition for Lemma~\ref{newlemma}, which completes the proof. 
\end{proof}

\section{Condition (D) implies Condition (U)}\label{DtoU}

In this section we investigate the relation between Condition (D) and Condition (U).

\begin{lemma}\label{returndist}
If $\tau_{r} (x) = n$ for some $x$,  then we have 
$\Delta(T^{2n}) < r $.
\end{lemma}

\begin{proof}
Let $n = \tau_{r} (x)$ and $[a, b)$ be the maximal interval containing $x$ on which $T^{n}$ is continuous. 
Note that both $a$ and $b$ are 
either discontinuity points of $T^{n}$ or end points,
i.e., $a, b \in D (T^{n}) \cup \{ 0,1\}$.
If $b-a < r$, then the proof is completed.
Now assume that $b - a \ge r$.

Let $\delta =  T^{n} (x) - x$.  Clearly $|\delta| < r \le b-a$
and $T^{n} [a, b) = [a + \delta, b + \delta)$.
If $\delta >0$, then  $b - \delta \in [a, b)$. And  if $\delta <0$ then $a - \delta \in [a,b)$. 
Therefore, $T^n (b- \delta) = b \in D(T^n)$ or $T^n (a - \delta) = a \in D(T^n)$, 
yielding 
$$ b - \delta \text{ or } a - \delta \in T^{-n} (D(T^n)) \subset D(T^{2n}). $$
Hence, we have 
$$ \Delta(T^{2n}) \le | \delta | < r.$$
\end{proof}

\begin{theorem}
Condition (D) implies Condition (U)
\end{theorem}

\begin{proof}
Suppose that $T$ with Condition (A) (equivalently (D)) does not satisfy Condition (U).
In \cite{KimMarmi}, it is implicitly shown (Proposition 3.5 and Theorem 3.6 in \cite{KimMarmi}) that
the normalized return time $\frac{ \log \tau_r (x)}{- \log r} $ is uniformly bounded by a sequence that converges to 1.
Assume that there is a sequence $r_i \downarrow 0$ and $x_i$ such that $\tau_{r_i} (x_i) < r_i^{-t}$ for some $t < 1$. 
Let $n_i = \tau_{r_i} (x_i) < r_i^{-t}$.
Then by Lemma~\ref{returndist} we have 
$$\Delta(T^{2n_i}) < r_i < \left( \frac{1}{n_i} \right)^{\frac 1t} = 2^{\frac 1t} \left( \frac{1}{2n_i} \right)^{\frac 1t},$$
which contradicts Condition (D).
\end{proof}

\section{Condition (U) implies Condition (Z)}\label{UtoZ}

In this section, we show that Condition (U) is stronger than Condition (Z).
Let $n_k$ be the sequence of Zorich's acceleration defined in Section~\ref{sec_background}.

\begin{lemma}\label{lem_utz}
If $n_{k+1} - n_k \ge d-1$, then for some $x$ and $r < \lambda^*(n_{k})$ we have 
$$ \frac{\log \tau_{r} (x)}{-\log r} 
< \frac{\log \| Z(k) \|}{\log ( \| Z(k,k+1) \|/d - 2 ) + \log \| Z(k) \| }. 
$$ 
\end{lemma}

\begin{proof}
Let $\alpha \in \mathcal A$ be the winner of the arrows and $\mathcal A'$ be the set of the losers of the arrows in the path $\gamma(n_k, n_{k+1})$.
If $\pi^{(n_k)}_t (\alpha) = d$, then 
$ \mathcal A' = \{ \beta \in \mathcal A : \pi^{(n_k)}_b (\beta) > \pi^{(n_k)}_b (\alpha) \}$
and $\pi^{(n)}_b$ is the cyclic permutation on $\mathcal A'$ for $n_k \le n \le n_{k+1}$.
For each $\beta \in \mathcal A'$ put $h_\beta = Z_{\beta \alpha}(k,k+1)$,
the number of arrows, of which loser is $\beta \in \mathcal A'$, in the path $\gamma(n_k, n_{k+1})$.
Put $$h := \left \lfloor \frac{n_{k+1} - n_k}{|\mathcal A'|} \right \rfloor \ge 1.$$ 
Then $h \le h_\beta \le h+1$ for all $\beta \in \mathcal A'$ and
\begin{equation}\label{hbound}
\| Z(k,k+1) \| = d + n_{k+1} - n_k \le d + (h+1) \cdot |\mathcal A'|
< d(h +2). \end{equation}
Let 
$$r := \frac{\lambda_{\alpha} (n_k)}{h}
> \frac{\lambda_{\alpha} (n_k) - \lambda_{\alpha} (n_{k+1})}{h}
= \frac{\sum_{\beta \in \mathcal A'} h_\beta \lambda_\beta(n_k)}{h}
\ge \sum_{\beta \in \mathcal A'} \lambda_\beta(n_k). $$
Since
$$T(n_k) ( x) = x - \sum_{\beta \in \mathcal A'} \lambda_\beta(n_k) \text{ on } x \in I_{\alpha} (n_k), $$
we have by (\ref{partition})
$$\tau_r (x) \le Z_{\alpha} (k) \text{ on } x \in I_{\alpha} (n_k).$$
Since $\lambda_{\alpha} (n_k) < 1/Z_\alpha (k)$ from (\ref{QQ}),
we have for $x \in I_{\alpha} (n_k)$
\begin{equation*}
\frac{\log \tau_{r} (x)}{-\log r} \le \frac{\log Z_{\alpha} (k)}{ \log h - \log \lambda_{\alpha} (n_k)} < \frac{\log Z_{\alpha} (k)}{ \log h + \log Z_{\alpha} (k)} \le \frac{\log \|Z (k) \| }{ \log h + \log \| Z (k)\|}. 
\end{equation*}
Therefore, by (\ref{hbound}), we have for $x \in I_{\alpha} (n_k)$
\begin{equation*}
\frac{\log \tau_{r} (x)}{-\log r} < \frac{\log \| Z(k) \|}{\log ( \| Z(k,k+1) \|/d - 2 ) + \log \| Z(k) \| }. 
\end{equation*}

If $\pi^{(n_k)}_b (\alpha) = d$, then 
we have the same bounds for $h$, $\lambda_\alpha (n_k)$ and 
$$T(n_k) (x) = x + \sum_{\beta \in \mathcal A'} \lambda_\beta(n_k) \text{ on } x \in I_{\alpha} (n_k). $$
Thus, we have the same inequality. 
\end{proof}

\begin{theorem}
Condition (U) implies Condition (Z)
\end{theorem}

\begin{proof}
Let $T$ be an interval exchange map without Condition (Z).
Then there are constants $t >0$ and $C$ such that for infinitely many $k$
\begin{equation}\label{notZ}
\| Z(k,k+1)\| \ge C \| Z(k) \|^t.
\end{equation}
Then for $k$ satisfying (\ref{notZ}),
 by Lemma~\ref{lem_utz}, there are $x$ and $r < \lambda^*(n_k)$ such that
\begin{equation*}
\frac{\log \tau_{r} (x)}{-\log r} 
< \frac{\log \| Z (k) \|}{ (1 + t) \log \| Z (k) \| + \log ( \frac Cd - 2\| Z (k) \|^{-t})}. 
\end{equation*}
Therefore we have sequences $\{x_i\}$ and $\{r_i\}$ such that $r_i \to 0$ and 
$$ \liminf_{i \to \infty} \frac{\log \tau_{r_i} (x_i)}{-\log r_i} \le \frac{1}{1+t} < 1,$$
which contradicts (U).
\end{proof}

\section{3-interval exchange maps}\label{3iem_sec}

In this section, we show that Condition (U), (R) and (Z) are equivalent for 3-interval exchange maps.
Let $T$ be a 3-interval exchange map with length data $(\lambda_A,\lambda_B,\lambda_C)$.
We may assume that $\pi_t(A) = 1, \pi_t(B) = 2, \pi_t(C) = 3$
and $\pi_b(C) = 3, \pi_b(B) = 2, \pi_b(A) = 1$.
Let $\lambda^* = \lambda(A) + \lambda(B) + \lambda(C) = 1$.

Define an irrational rotation $\bar T$ on
$ \bar I = [0, \lambda^*+\lambda_B )$
by
\begin{equation}\label{bart}
\bar T(x) = \begin{cases}
x + \lambda_B + \lambda_C, &\text{ if } x + \lambda_B + \lambda_C \in \bar I, \\
x + \lambda_B + \lambda_C - (\lambda^* + \lambda_B), &\text{ if } x + \lambda_B + \lambda_C \notin \bar I.
\end{cases}\end{equation}
Then $\bar T$ is a 2-interval exchange map (irrational rotation) with length data $(\lambda_{\bar A}, \lambda_{\bar C})$,
where $\lambda_{\bar A} = \lambda_A + \lambda_B$ and $\lambda_{\bar C} = \lambda_A + \lambda_C$.
Note that $T$ is the induced map of $\bar T$ on $[0, \lambda^*)$ and 
$T$ satisfies the Keane property if and only if the rotation $\bar T$ is irrational.

Let $\alpha = \frac{\lambda_B + \lambda_C}{\lambda^*+\lambda_B}$ be the rotation angle of $\bar T$
and let $a_k$ and $p_k/q_k$ be the partial quotients and partial convergents of $\alpha$.

\begin{lemma}[Denjoy-Koksma inequality (see \cite{Herman})]\label{Denjoy-Koksma}
Let $\bar T$ be an irrational rotation by $\alpha$ with partial quotient denominators $q_k$ and $f$ be a real valued function of bounded variation on the unit interval. 
Then for any $x$ we have 
$$
\left | \sum_{i=0}^{q_k-1} f (\bar T^i x) - q_k \int f d\mu \right| < \mathrm{var} (f).
$$
\end{lemma}

\begin{proposition}\label{3iemprop}
Let $T$ be a 3-interval exchange map $T$ on $[0, \lambda^*)$ and $\bar T$ be the inducing rotation, defined as (\ref{bart}).
For any $x \in [0, \lambda^*)$ we have
$$ \lim_{r \to 0^+} \frac{ \log \tau_{r} (x) }{ - \log r} = 1 $$
if and only if
$$ \lim_{r \to 0^+} \frac{ \log \bar \tau_{r} (x) }{ - \log r} = 1,  $$
where $\bar \tau_r $ is the first return time of $\bar T$.
\end{proposition}

\begin{proof}
Since $T$ is the induced map of $\bar T$ on $[0,\lambda^*)$, 
if ${\bar T}^{\bar m} (x) \in [0, \lambda^*)$ for $x \in [0, \lambda^*)$, we have
$$
{\bar T}^{\bar m} (x)= T^m(x) , \text{ where } m = \sum_{i=0}^{\bar m -1} 1_{[0, \lambda^*)} (\bar T^i (x)).
$$
Thus, for $x \in [0, \lambda^* - r)$
we have ${\bar T}^{\bar \tau_r(x)} (x) \in [0, \lambda^*)$ and 
$$
\tau_r (x) = \sum_{i=0}^{\bar \tau_r (x) -1} 1_{[0, \lambda^*)} (\bar T^i (x)).
$$
Let $q_k$ be the partial denominators of $\alpha$, the rotational angle of $\bar T$.
Then clearly $\bar \tau_r (x) = q_k $ for some $k \ge 0$.
From Lemma~\ref{Denjoy-Koksma}, we have
$$ \left|\tau_r (x) - \frac{\bar \tau_r (x)}{1+\lambda_B} \right | <  2, $$
which completes the proof immediately.
\end{proof}

As a corollary, a 3-interval exchange map $T$ satisfies Condition (U) if and only if $\bar T$ is of Roth's type. 
Moreover, we see that Condition (R) and Condition (U) are equivalent for 3-interval exchange maps.

Now we compare the Rauzy-Veech induction algorithm for $T$ and $\bar T$.
There are 6 arrows in the Rauzy diagram for a 3-interval exchange map $T$ (see Figure~\ref{diagram3}).
Each arrow in the Rauzy diagram for $\bar T$ corresponds to two arrows of the same loser in the Rauzy diagram for $T$ and remaining 2 arrows of the loser $B$ are not be mapped to any arrows in the Rauzy diagram for $\bar T$ (Figure~\ref{induced_diagram3}).
Denote an arrow of the Rauzy diagram for 2 or 3-interval exchange map by $\alpha(\beta)$, 
where $\alpha$ is the winner and $\beta$ is the loser of the arrow. 

\begin{figure}\caption{3-interval exchange}\label{diagram3}
\includegraphics{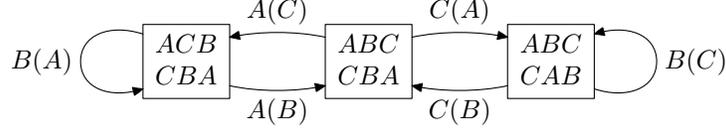}
\end{figure}

\begin{figure}\caption{induced Rauzy diagram for 3-interval exchange map}\label{induced_diagram3}
\includegraphics{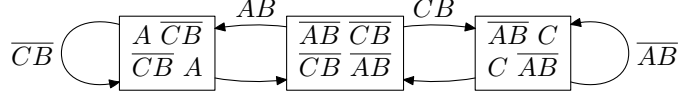}
\end{figure}

Let $$ R(n) : = \begin{bmatrix} R_{A \bar A}(n) & R_{A \bar C}(n) \\ R_{B \bar A}(n) & R_{B \bar C}(n) \\ R_{C \bar A}(n) & R_{C \bar C}(n) \end{bmatrix}  = Q(n) \begin{bmatrix} 1 & 0 \\ 1 & 1 \\ 0 & 1 \end{bmatrix}.$$
\begin{lemma}\label{3iemsimple}
For $n \ge 0$ we have 
\begin{equation*}
[ 0, 1, 0] R(n) = 
\begin{cases}
[ 1, 0, 1] R(n), &\pi^{(n)} = \bigl( \begin{smallmatrix} A & B & C \\ C & B & A \end{smallmatrix} \bigr), \\
[ 0, 0, 1] R(n), &\pi^{(n)} = \bigl( \begin{smallmatrix} A & C & B \\ C & B & A \end{smallmatrix} \bigr), \\
[ 1, 0, 0] R(n), &\pi^{(n)} = \bigl( \begin{smallmatrix} A & B & C \\ C & A & B \end{smallmatrix} \bigr).
\end{cases}
\end{equation*}
\end{lemma}

\begin{proof}
By the symmetry we only consider arrows of $A(B), A(C), B(A)$.
When $\gamma^T(n,n+1) = A(B)$, we have $\pi^{(n)} =\bigl( \begin{smallmatrix} A & C & B \\ C & B & A \end{smallmatrix}  \bigr)$, $\pi^{(n+1)} = \bigl( \begin{smallmatrix} A & B & C \\ C & B & A \end{smallmatrix} \bigr)$. Thus 
\begin{equation*}\begin{split}
[0,1,0] R(n+1) &= [0,1,0] \begin{bmatrix} 1 & 0 & 0 \\ 1 & 1 & 0 \\ 0 & 0 & 1 \end{bmatrix} R(n) = [1,1,0] R(n) \\
&= [1,0,1] R(n) = [1,0,1] \begin{bmatrix} 1 & 0 & 0 \\ 1 & 1 & 0 \\ 0 & 0 & 1 \end{bmatrix} R(n) = [1,0,1] R(n+1).
\end{split}\end{equation*}
By an arrow $\gamma^T(n,n+1)=A(C)$, we have $\pi^{(n)} =\bigl( \begin{smallmatrix} A & B & C \\ C & B & A \end{smallmatrix} \bigr)$, $\pi^{(n+1)} = \bigl( \begin{smallmatrix} A & C & B \\ C & B & A \end{smallmatrix}  \bigr)$, so 
\begin{equation*}\begin{split}
[0,1,0] R(n+1) &= [0,1,0] \begin{bmatrix} 1 & 0 & 0 \\ 0 & 1 & 0 \\ 1 & 0 & 1 \end{bmatrix} R(n) = [0,1,0] R(n) \\
&= [1,0,1] R(n)= [0,0,1] \begin{bmatrix} 1 & 0 & 0 \\ 0 & 1 & 0 \\ 1 & 0 & 1 \end{bmatrix} R(n)= [0,0,1] R(n+1).
\end{split}\end{equation*}
If $\gamma^T(n,n+1)=B(A)$, then 
we have $\pi^{(n)} = \pi^{(n+1)} = \bigl( \begin{smallmatrix} A & C & B \\ C & B & A \end{smallmatrix}  \bigr)$ and
\begin{equation*}\begin{split}
[0,1,0] R(n+1) &= [0,1,0] \begin{bmatrix} 1 & 1 & 0 \\ 0 & 1 & 0 \\ 0 & 0 & 1 \end{bmatrix} R(n) = [0,1,0] R(n) \\
&= [0,0,1] R(n) = [0,0,1] \begin{bmatrix} 1 & 1 & 0 \\ 0 & 1 & 0 \\ 0 & 0 & 1 \end{bmatrix} R(n) = [0,0,1] R(n+1).
\end{split}\end{equation*}
Since the lemma holds for $n = 0$, the induction rule completes the proof. 
\end{proof}

For a given 3-interval exchange map $T$ define
$$\ell(n) := \# \left\{ 1 \le m \le n \mid \pi^{(m)} = \bigl( \begin{smallmatrix} A & C & B \\ C & B & A \end{smallmatrix} \bigr) \text{ or } \bigl( \begin{smallmatrix} A & B & C \\ C & A & B \end{smallmatrix} \bigr) \right\}.$$

\begin{proposition}\label{3iem_prop}
Let $T$ be a 3-interval exchange map $T$. 
By the mapping $\alpha(A) \mapsto \bar C (\bar A)$, $\alpha(C) \mapsto \bar A (\bar C)$ and $\alpha(B) \mapsto \varepsilon$, where $\varepsilon$ is the empty arrow,
the infinite sequence of arrows in the Rauzy diagram for $T$ is mapped to the infinite sequence of arrows in the Rauzy diagram for $\bar T$.

Denote by $\bar Q(m) = B_{\gamma^{\bar T}(0,m)}$ be the continued fraction matrix for $\bar T$.
Then for $n \ge 0$
$$ \bar Q (\ell(n)) = \begin{bmatrix} 1 & 0 & 0 \\ 0 & 0 & 1 \end{bmatrix} R( n ) \ \text{ and }
$$
$$ \bar \lambda (\ell(n)) =
\lambda (n) \begin{bmatrix} 1 & 0 \\ 1 & 1 \\ 0 & 1 \end{bmatrix}, \quad  \lambda (n) \begin{bmatrix} 1 & 0 \\ 0 & 1 \\ 0 & 1 \end{bmatrix}, \quad \lambda (n) \begin{bmatrix} 1 & 0 \\ 1 & 0 \\ 0 & 1 \end{bmatrix}$$
for $\pi^{(n)} = \bigl( \begin{smallmatrix} A & B & C \\ C & B & A \end{smallmatrix} \bigr)$,
$\bigl( \begin{smallmatrix} A & C & B \\ C & B & A \end{smallmatrix} \bigr)$,
$\bigl( \begin{smallmatrix} A & B & C \\ C & A & B \end{smallmatrix} \bigr)$, respectively.
\end{proposition}

\begin{proof}
By the symmetry we only consider arrows of $A(B), A(C), B(A)$.

Case (i) : If  $\gamma^T(n,n+1) = A(B)$, then the corresponding arrow of the Rauzy map for $\bar T$ is empty, i.e., $\ell(n+1) = \ell(n)$.
Since
$$ Q(n,n+1) = \begin{bmatrix} 1 & 0 & 0 \\ 1 & 1 & 0 \\ 0 & 0 & 1 \end{bmatrix} \text{ and }
\pi^{(n)} =\bigl( \begin{smallmatrix} A & C & B \\ C & B & A \end{smallmatrix}  \bigr), \pi^{(n+1)} = \bigl( \begin{smallmatrix} A & B & C \\ C & B & A \end{smallmatrix} \bigr), $$
we have 
\begin{equation*}\begin{split}
\bar \lambda (\ell(n)) 
&= \lambda (n)  \begin{bmatrix} 1 & 0 \\ 0 & 1 \\ 0 & 1 \end{bmatrix} 
= \lambda (n+1)  \begin{bmatrix} 1 & 0 & 0 \\ 1 & 1 & 0 \\ 0 & 0 & 1 \end{bmatrix} \begin{bmatrix} 1 & 0 \\ 0 & 1 \\ 0 & 1 \end{bmatrix}  
= \lambda (n+1)  \begin{bmatrix} 1 & 0 \\ 1 & 1 \\ 0 & 1 \end{bmatrix}  \\
&= \bar \lambda (\ell(n+1)),  
\end{split}\end{equation*}
\begin{equation*}\begin{split}
\bar Q(\ell(n+1)) &= \begin{bmatrix} 1 & 0 & 0 \\ 0 & 0 & 1 \end{bmatrix} R(n+1)
= \begin{bmatrix} 1 & 0 & 0 \\ 0 & 0 & 1 \end{bmatrix} \begin{bmatrix} 1 & 0 & 0 \\ 1 & 1 & 0 \\ 0 & 0 & 1 \end{bmatrix} R(n) \\
&= \begin{bmatrix} 1 & 0 & 0 \\ 0 & 0 & 1 \end{bmatrix} R(n) = \bar Q(\ell(n)).
\end{split}\end{equation*}

Case (ii) : If $\gamma^T(n,n+1)=A(C)$, then $\gamma^{\bar T}(\ell(n),\ell(n+1))=\bar A(\bar C)$ and
$$ Q(n,n+1) = \begin{bmatrix} 1 & 0 & 0 \\ 0 & 1 & 0 \\ 1 & 0 & 1 \end{bmatrix}, \quad \bar Q(\ell(n), \ell(n+1)) = \begin{bmatrix} 1 & 0 \\ 1 & 1 \end{bmatrix}.$$  
Since $\pi^{(n)} =\bigl( \begin{smallmatrix} A & B & C \\ C & B & A \end{smallmatrix} \bigr)$, $\pi^{(n+1)} = \bigl( \begin{smallmatrix} A & C & B \\ C & B & A \end{smallmatrix} \bigr)$, we have
\begin{equation*}\begin{split}
\bar \lambda (\ell(n)) 
&= \lambda (n)  \begin{bmatrix} 1 & 0 \\ 1 & 1 \\ 0 & 1 \end{bmatrix} 
= \lambda (n+1)  \begin{bmatrix} 1 & 0 & 0 \\ 0 & 1 & 0 \\ 1 & 0 & 1 \end{bmatrix} \begin{bmatrix} 1 & 0 \\ 1 & 1 \\ 0 & 1 \end{bmatrix}  
= \lambda (n+1)  \begin{bmatrix} 1 & 0 \\ 0 & 1 \\ 0 & 1 \end{bmatrix} \begin{bmatrix} 1 & 0 \\ 1 & 1 \end{bmatrix}  \\
&= \bar \lambda (\ell(n+1))   \begin{bmatrix} 1 & 0 \\ 1 & 1 \end{bmatrix},
\end{split}\end{equation*}
\begin{equation*}\begin{split}
\bar Q(\ell(n+1)) &= \begin{bmatrix} 1 & 0 & 0 \\ 0 & 0 & 1 \end{bmatrix} R(n+1)
= \begin{bmatrix} 1 & 0 & 0 \\ 0 & 0 & 1 \end{bmatrix} \begin{bmatrix} 1 & 0 & 0 \\ 0 & 1 & 0 \\ 1 & 0 & 1 \end{bmatrix} R(n) \\
& =\begin{bmatrix} 1 & 0 \\ 1 & 1 \end{bmatrix} \begin{bmatrix} 1 & 0 & 0 \\ 0 & 0 & 1 \end{bmatrix} R(n) 
= \begin{bmatrix} 1 & 0 \\ 1 & 1 \end{bmatrix} \bar Q(\ell(n)).
\end{split}\end{equation*}

Case (iii) : If $\gamma^T(n,n+1)=B(A)$, then 
$\gamma^{\bar T}(\ell(n),\ell(n+1))=\bar A(\bar C)$ and
$$ Q(n,n+1) = \begin{bmatrix} 1 & 1 & 0 \\ 0 & 1 & 0 \\ 0 & 0 & 1 \end{bmatrix}, \quad \bar Q(\ell(n), \ell(n+1)) = \begin{bmatrix} 1 & 1 \\ 0 & 1 \end{bmatrix}.$$  
From the fact that $\pi^{(n)} = \pi^{(n+1)} = \bigl( \begin{smallmatrix} A & C & B \\ C & B & A \end{smallmatrix}  \bigr)$ we have
\begin{equation*}\begin{split}
\bar \lambda (\ell(n)) 
&= \lambda (n)  \begin{bmatrix} 1 & 0 \\ 0 & 1 \\ 0 & 1 \end{bmatrix} 
= \lambda (n+1)  \begin{bmatrix} 1 & 1 & 0 \\ 0 & 1 & 0 \\ 0 & 0 & 1 \end{bmatrix} \begin{bmatrix} 1 & 0 \\ 0 & 1 \\ 0 & 1 \end{bmatrix}  
= \lambda (n+1)  \begin{bmatrix} 1 & 0 \\ 0 & 1 \\ 0 & 1 \end{bmatrix} \begin{bmatrix} 1 & 1 \\ 0 & 1 \end{bmatrix} \\  
&= \bar \lambda (\ell(n+1))   \begin{bmatrix} 1 & 1 \\ 0 & 1 \end{bmatrix}  
\end{split}\end{equation*}
and by Lemma~\ref{3iemsimple}
\begin{equation*}\begin{split}
\bar Q(\ell(n+1)) &= \begin{bmatrix} 1 & 0 & 0 \\ 0 & 0 & 1 \end{bmatrix} R(n+1)
= \begin{bmatrix} 1 & 0 & 0 \\ 0 & 0 & 1 \end{bmatrix} \begin{bmatrix} 1 & 1 & 0 \\ 0 & 1 & 0 \\ 0 & 0 & 1 \end{bmatrix} R(n) \\
&= \begin{bmatrix} 1 & 1 & 0 \\ 0 & 0 & 1 \end{bmatrix} R(n)
= \begin{bmatrix} 1 & 0 & 1 \\ 0 & 0 & 1 \end{bmatrix} R(n) 
= \begin{bmatrix} 1 & 1 \\ 0 & 1 \end{bmatrix} \bar Q(\ell(n)).
\end{split}\end{equation*}

Since the proposition holds for $n = 0$, the induction rule completes the proof. 
\end{proof}

We have the following inequality for $\|Q(n)\|$:

\begin{lemma}\label{3iem_lem1}
We have 
$$\frac 12 \| \bar Q (s_n) \| \le \| Q(n)\| \le 2 \| \bar Q (s_n) \| .$$
\end{lemma}

\begin{proof}
By Proposition~\ref{3iem_prop} we have
$$
\| \bar Q (\ell(n)) \| = 
\left \| \begin{bmatrix} 1 & 0 & 0 \\ 0 & 0 & 1 \end{bmatrix} R( n ) \right \| 
\le \| R(n) \| = \left \| Q(n) \begin{bmatrix} 1 & 0 \\ 1 & 1 \\ 0 & 1 \end{bmatrix} \right \| 
\le 2 \|  Q(n) \|.  
$$
For the other side from Lemma~\ref{3iemsimple}
\begin{equation*}
\begin{split}
\| Q (n) \| &\le \| R (n) \| = \| [ 1,1,1]  R(n)  \| =   \| [1,0,1] R(n) \|  + \| [0,1,0] R(n) \|  \\
&\le 2  \| [1,0,1] R(n) \| = 2 \left \| \begin{bmatrix} 1 & 0 & 0 \\ 0 & 0 & 1 \end{bmatrix} R(n)  \right \| = 2 \| \bar Q(s_n)\|.  
\end{split}
\end{equation*}
\end{proof}

Let $n_k$ and $\bar n_k$ be the sequence of Zorich's acceleration for $T$ and $\bar T$ respectively as defined in Section 2.
Also denote by $\bar Z(k)$ be the Zorich's acceleration matrix for $\bar T$.

\begin{lemma}\label{3iem_lem2}
For each $k \ge 0$, There exist $j(k) \ge 0$ such that 
\begin{equation*}
\ell(n_{j(k)}) = \bar n_k,
\end{equation*}
Moreover, we have $j(k+1) \le j(k) + 3$ and 
$$\| \bar Z ( k, k +1) \| < \| Z ( j(k), j(k+1) ) \| \le 2 \| \bar Z ( k, k +1) \|.$$
\end{lemma}

\begin{proof}
For $k=0$, we have $n_0 = 0$ and $j(0) = 0$.
Suppose that for a given $k\ge 0$ there exists $j(k)$ satisfying $\ell(n_{j(k)}) = \bar n_k$.
For $\gamma^{\bar T}(\bar n_k, \bar n_{k+1}) = \bar A(\bar C)^w$, $w = \bar n_{k+1} - \bar n_k \ge 1$
there are 7 cases of $\gamma^T ( n_{j(k)}, \infty )$:
\begin{align*}
&A(B)A(C) \cdots A(B)A(C) \ B(A) \cdots, & j(k+1) = j(k)+1, \\
&A(B)A(C) \cdots A(C)A(B) \ C(A) \cdots, & j(k+1) = j(k)+1, \\
&A(C)A(B) \cdots A(C)A(B) \ C(A) \cdots, & j(k+1) = j(k)+1, \\
&A(C)A(B) \cdots A(B)A(C) \ B(A) \cdots, & j(k+1) = j(k)+1, \\
&B(C) \cdots B(C) \ C(B) C(A) \cdots, & j(k+1) = j(k)+1, \\
&B(C) \cdots B(C) \ C(B) \ A(C) A(B) \cdots A(C) A(B) \ C(A) \cdots, & j(k+1) = j(k)+3, \\
&B(C) \cdots B(C) \ C(B) \ A(C) A(B) \cdots A(B) A(C) \ B(A) \cdots, & j(k+1) = j(k)+3,
\end{align*}
Moreover, $Z(j(k),j(k+1))$ is
\begin{multline*}
\begin{bmatrix} 1 & 0 & 0 \\ w & 1 & 0 \\ w & 0 & 1 \end{bmatrix}, 
\begin{bmatrix} 1 & 0 & 0 \\ w+1 & 1 & 0 \\ w & 0 & 1 \end{bmatrix}, 
\begin{bmatrix} 1 & 0 & 0 \\ w & 1 & 0 \\ w & 0 & 1 \end{bmatrix}, 
\begin{bmatrix} 1 & 0 & 0 \\ w-1 & 1 & 0 \\ w & 0 & 1 \end{bmatrix}, \\
\begin{bmatrix} 1 & 0 & 0 \\ 0 & 1 & 0 \\ 0 & w & 1 \end{bmatrix}, 
\begin{bmatrix} 1 & 0 & 0 \\ w_2 & w_1+1 & 0 \\ w_2 & w_1 & 1 \end{bmatrix}, 
\begin{bmatrix} 1 & 0 & 0 \\ w_2 -1 & w_1+1 & 0 \\ w_2 & w_1 & 1 \end{bmatrix},
\end{multline*}
respectively, according to 7 cases of the path $\gamma^T ( n_{j(k)}, n_{j(k+1)} )$.
Here $w = w_1 + w_2$, $w_1 \ge 1$, $w_2 \ge 1$.
Compared with $\bar Z(k, k+1) = \begin{pmatrix} 1 & 0 \\ w & 1 \end{pmatrix},$
we have 
$$\| \bar Z(k, k+1) \| = w + 2 < \| Z(j(k),j(k+1)) \| \le 2w + 4 = 2 \| \bar Z(k, k+1) \|.$$

By the symmetry we have the same inequality for $\gamma^{\bar T}(\bar n_k, \bar n_{k+1}) = \bar C(\bar A)^{\bar n_{k+1}-\bar n_k}$.
\end{proof}

\begin{theorem}\label{3iem_th}
The 3-interval exchange map $T$ satisfies Condition (Z) if and only if 
the irrational rotation $\bar T$, which induces $T$, is of Roth's type.
\end{theorem}

\begin{proof}
Suppose that the 3-interval exchange map $T$ satisfies Condition (Z). 
Then for any $\varepsilon >0$ we have $C_\varepsilon >0$ such that
$ \| Z(k,k+1) \| \le C_\varepsilon \| Z(k)\|^\varepsilon. $
Therefore we have by Lemma~\ref{3iem_lem2}
\begin{align*}
\| \bar Z(k,k+1) \| &< \| Z(j(k),j(k+1))\| \le \| Z(j(k),j(k)+3)\| \\
&\le \| Z(j+2,j+3) \| \cdot \| Z(j+1,j+2) \| \cdot \| Z(j+1,j) \| \\
&\le C_\varepsilon^3 \| Z(j+2) )\|^\varepsilon \cdot \| Z(j+1)\|^\varepsilon \cdot \| Z(j)\|^\varepsilon \\
&\le C_\varepsilon^{3+3\varepsilon+\varepsilon^2} \| Z(j(k))\|^{3\varepsilon + 3 \varepsilon^2 + \varepsilon^3} 
\le \left(2^\varepsilon C_\varepsilon \right)^{3+3\varepsilon+\varepsilon^2} \| \bar Z(k)\|^{3\varepsilon + 3 \varepsilon^2 + \varepsilon^3},
\end{align*}
where the last inequality is from Lemma~\ref{3iem_lem1}.

For the opposite direction we assume that $\bar T$ is of Roth's type:
For any $\varepsilon >0$ there is $\bar C_\varepsilon >0$ such that
$\| \bar Z(k',k'+1) \| \le C_\varepsilon \| \bar Z(k')\|^\varepsilon. $
For each $k$, we can find $k'$ such that $j(k') \le k < k+1 \le j(k' +1)$. 
Therefore we have by Lemma~\ref{3iem_lem2} and \ref{3iem_lem1}
\begin{align*}
\| Z(k, k+1) \| &\le \| Z(j(k'), j(k'+1)) \| \le 2 \| \bar Z(k',k'+1) \| \\
&\le \bar C_\varepsilon \| \bar Z(k')\|^\varepsilon 
\le 2^\varepsilon \bar C_\varepsilon \| Z(k)\|^\varepsilon. 
\end{align*}

\end{proof}

\section{Example with Condition (R) without Condition (Z)}\label{exm_sec1}

In this section, we discuss an example of 4 interval exchange map such that satisfies Condition (R) but not Condition (Z).

\begin{figure}\caption{Rauzy diagram for the example}\label{diagram4}
\includegraphics{graphic_47.mps} 
\end{figure}

Let $T$ be a 4-interval exchange map with the permutation data $\pi^{(0)} = \bigl( \begin{smallmatrix} A&B&D&C\\D&A&C&B \end{smallmatrix}\bigr)$.
Assume that the length data of $T$ is determined by the infinite path in the Rauzy diagram, denoted by the winner of each arrow (see Figure~\ref{diagram4})
$$ C^{s_1} B \left( D^2 A^3 D \right)^{2+1} B  \cdot
C^{s_2} B \left( D^2 A^3 D \right)^{2^2+2} B \cdots C^{s_k} B \left( D^2 A^3 D \right)^{2^k+k} B \cdots. $$

Let $$\ell_k = \sum_{i=1}^k (s_i + 6 \cdot 2^i+i + 2), \ \ell_0 = 0, \
\text{ and } \
s_k = F_{2^{k+1}}.$$
The matrix associated to the path $ C^{s_k} B \left( D^2 A^3 D\right)^{2^k+k} B $ is
\begin{align*}
Q(\ell_{k-1},\ell_k) &= 
\begin{bmatrix} 1 & 0 & 0 & 0 \\ 0 & 1 & 0 & 0 \\ 0 & 0 & 1 & 0 \\ 0 & 1 & 0 & 1 \end{bmatrix} 
\begin{bmatrix} 2 & 0 & 0 & 1 \\ 1 & 1 & 0 & 1 \\ 1 & 0 & 1 & 1 \\ 1 & 0 & 0 & 1 \end{bmatrix}^{2^k+k}  
\begin{bmatrix} 1 & 0 & 0 & 0 \\ 0 & 1 & s_k & 0 \\ 0 & 1 & s_k + 1 & 0 \\ 0 & 0 & 0 & 1 \end{bmatrix}  \\
&=\begin{bmatrix} F_{2^{k+1} + 2k+1} & 0 &0 &F_{2^{k+1} + 2k} \\
F_{2^{k+1} + 2k+1}-1 & 1 & F_{2^{k+1}} & F_{2^{k+1} + 2k} \\
F_{2^{k+1} + 2k+1}-1 & 1 & F_{2^{k+1}} + 1 & F_{2^{k+1} + 2k}  \\
F_{2^{k+1} + 2k+2}-1 & 1 & F_{2^{k+1}} & F_{2^{k+1} + 2k+1} \end{bmatrix},
\end{align*}
where $F_n$ is the Fibonacci sequence: $F_{-1}=1, F_0 = 0, F_{n+1} = F_n + F_{n-1}$.
Note that $F_n = \frac{1}{\sqrt 5} (g^{n} - (-g)^{-n})$, $g = \frac{\sqrt 5 + 1}{2}$. 

The following lemma provides a rough but useful estimate on the relative size as well as on the growth rate  of the sequence 
$(Q_\alpha(\ell_k))_{k \ge 1}$.

\begin{lemma}\label{l1}
For all  $k \ge 1$ we have 
$$1 < \frac{Q_D (\ell_k)}{g Q_A(\ell_k)} < \frac{Q_B (\ell_k)}{Q_A(\ell_k)} <\frac{Q_C (\ell_k)}{Q_A(\ell_k)} <  1 + \frac{1}{g^{2k+1}}. $$
$$ Q_A (\ell_k) \le g^{2^{k+1}+2k+1} Q_A (\ell_{k-1})$$
\end{lemma}

\begin{proof}
Let 
$$ \frac{Q_B (\ell_k)}{Q_A (\ell_k)} = 1 + r_B (k),
\quad \frac{Q_B (\ell_k)}{Q_A (\ell_k)} =  1 + r_C (k),
\quad \frac{Q_D (\ell_k)}{gQ_A (\ell_k)} =  1 + r_D (k).$$
Then 
$$ r_B(1) = \frac{F_4}{F_8}, \quad r_C(1) = \frac{F_4+1}{F_8}, \quad r_D(1) = \frac{F_9}{gF_8} - 1 + \frac{F_4}{gF_8}$$ 
so by simple calculations
$$ 0 < r_D(1) < r_B(1) < r_C(1) < \frac{1}{g^3}.$$ 
If $0 < r_D(k-1) < r_B(k-1) < r_C(k-1) < \frac{1}{g^{2k-1}}$, then using
\begin{equation*}\begin{split}
Q_\alpha (\ell_k) &= \sum_{\gamma} Q_{\alpha\gamma} (\ell_k) = \sum_{\gamma} \sum_\beta Q_{\alpha\beta} (\ell_{k-1},\ell_k) Q_{\beta\gamma}  (\ell_{k-1}) \\
&= \sum_{\beta} Q_{\alpha\beta}(\ell_{k-1},\ell_k) Q_\beta (\ell_{k-1}),
\end{split}\end{equation*}
we have
\begin{equation*}
\begin{split}
0 < r_D(k) < r_B(k) < r_C(k) &= \frac{r_B(k-1) + (F_{2^{k+1}} +1)(1+ r_C(k-1))}{F_{2^{k+1}+2k+1} + gF_{2^{k+1}+2k} ( 1 + r_D(k-1))} \\
&< \frac{F_{2^{k+1}} + r_C(k-1)F_{2^{k+1}} + 2}{ 2gF_{2^{k+1}+2k}}
 < \frac{1}{g^{2k+1}}
\end{split}
\end{equation*}

We also have for $k \ge 1$
\begin{equation*}
\begin{split}
Q_A (\ell_k) &= \left( F_{2^{k+1}+2k+1} + gF_{2^{k+1}+2k} + \frac{F_{2^{k+1}+2k}}{g^{2k}} \right)Q_A (\ell_{k-1}) \\
&< \frac{1}{\sqrt 5} \left( 2g^{2^{k+1}+2k+1} + g^{2^{k+1}} \right) Q_A(\ell_{k-1})
< g^{2^{k+1}+2k+1} Q_A(\ell_{k-1}).
\end{split}
\end{equation*}
\end{proof}

By the previous lemma we have
\begin{align*}
\| Q(\ell_k) \| &< \left( 3 + g + \frac{3}{g^{2k}} \right) Q_A(\ell_k) 
< g^4 \cdot g^{2^{k+1}+2k+1} \cdots g^{2^{2}+2+1} Q_A(\ell_0) \\
&< g^{2^{k+2}+k(k+1)+k}.
\end{align*}
Since
$$ Q(\ell_k, \ell_k+s_{k+1}) = \begin{bmatrix} 1 & 0 & 0 & 0 \\ 0 & 1 & s_{k+1} & 0 \\ 0 & 0 & 1 & 0 \\ 0 & 0 & 0 & 1 \end{bmatrix},$$
for large $k$
$$ \| Q(\ell_k) \|^{1/2} < g^{2^{k+1} +k^2/2 + 2k}
\le \frac{g^{2^{k+2}}}{\sqrt{5}} < F_{2^{k+2}} + 4 = \| Q(\ell_k, \ell_k+s_{k+1})\|,$$
which implies that this interval exchange map $T$ does not satisfy Condition (Z).

Since $ \lambda(\ell_{k+1})  Q(\ell_{k}, \ell_{k+1} ) = \lambda(\ell_k)$, 
the length data of $T(\ell_k)$, 
$\lambda(\ell_k)$ is a vector in the simplex with the vertexes
\begin{align*}
\lambda^*(\ell_{k+1}) & \bigl[ F_{2^{k+2} + 2k +3} , 0 , 0 , F_{2^{k+2} + 2k+2} \bigr], \\
\lambda^*(\ell_{k+1}) & \bigl[ F_{2^{k+2} + 2k+3}-1, 1 , F_{2^{k+2}} , F_{2^{k+2} + 2k+2} \bigr], \\
\lambda^*(\ell_{k+1}) & \bigl[ F_{2^{k+2} + 2k+3}-1 , 1 , F_{2^{k+2}}+1 , F_{2^{k+2} + 2k+2} \bigr],  \\
\lambda^*(\ell_{k+1}) & \bigl[F_{2^{k+2} + 2k+4}-1 , 1 , F_{2^{k+2}} , F_{2^{k+2} + 2k+3} \bigr]. 
\end{align*}
Therefore we have
$$ \lambda^*(\ell_{k}) < \left( F_{2^{k+2}+2k+5} + F_{2^{k+2}} \right) \lambda^*(\ell_{k+1}) < g^{2^{k+2}+2k+3} \lambda^*(\ell_{k+1}) ,
$$
$$ \lambda_B(\ell_k) < \lambda^*(\ell_{k+1}) < \lambda_C (\ell_k) < \lambda_D(\ell_k) < \lambda_A(\ell_k)$$
and
\begin{equation*}
\lambda_C(\ell_k) <  \frac{(F_{2^{k+2}}+1)\lambda^*(\ell_{k})}{F_{2^{k+2} +2k+4}+F_{2^{k+2}}+1} < \frac{\lambda^*(\ell_{k})}{g^{2k+3}},
\end{equation*}
$$\frac{1}{g^3} < \frac{F_{2^{k+2} + 2k+2}}{F_{2^{k+2} + 2k+4}+F_{2^{k+2}} +1} < \frac{\lambda_D(\ell_k)}{\lambda^* (\ell_k)} < \frac{F_{2^{k+2} + 2k+2}}{F_{2^{k+2} + 2k+4}} < \frac{1}{g^2}.$$

Using the relation
$$\sum_\alpha \lambda_\alpha (\ell_k) Q_\alpha (\ell_k) = 1$$
we have by Lemma~\ref{l1}
\begin{equation}\label{qalambda}
\lambda^* (\ell_k) < \frac 1{Q_A(\ell_k)} < \lambda^* (\ell_k) + \frac{\lambda_D (\ell_k)}{g} + \frac{\lambda_B (\ell_k)+\lambda_C (\ell_k)+\lambda_D (\ell_k)}{g^{2k+1}} < g \lambda^* (\ell_k).
\end{equation}
We also have 
\begin{equation*}
\lambda_B(\ell_k)Q_B(\ell_k) < \lambda_C(\ell_k)Q_C(\ell_k) 
< \frac{\lambda^*(\ell_{k})}{g^{2k+3}} \left(1+ \frac{1}{g^{2k+1}} \right) Q_A(\ell_k) < \frac1{g^{2k+2}}.
\end{equation*}

By the permutation data $\pi^{(\ell_k)} = \bigl( \begin{smallmatrix} A&B&D&C\\D&A&C&B \end{smallmatrix}\bigr)$
we have
\begin{equation*}
T(\ell_k) (x) = 
\begin{cases} x + \lambda_D(\ell_k) &\text{ for } x \in I_A(\ell_k) , \\
x - \left( \lambda_A(\ell_k) +  \lambda_B(\ell_k) \right) &\text{ for } x \in I_D(\ell_k), \\
x  - \lambda_B (\ell_k) &\text{ for } x \in I_C (\ell_k).\end{cases}
\end{equation*}

Let $\tilde T_k$ be the 2-interval exchange map on $[0,\lambda_A(\ell_k)+\lambda_B(\ell_k)+\lambda_D(\ell_k) ) = [0,\lambda^*(\ell_k+s_{k+1}+1))$ with $\lambda_{\tilde A} (\ell_k) = \lambda_A (\ell_k) + \lambda_B (\ell_k)$ and $\lambda_{\tilde D} (\ell_k) = \lambda_D (\ell_k)$.
Then $$ T(\ell_k) (x) = T(\ell_k+s_{k+1}+1) (x) = \tilde T_k (x) \text{ on }  x \in I_A(\ell_k) \cup I_D(\ell_k).$$
\begin{lemma}\label{l}
If 
$$x \in \left( I_A(\ell_k) \cup I_D(\ell_k)\right) \setminus \left( \bigcup_{i=0}^{m} T(\ell_k)^{-i} I_B(\ell_k) \right), $$ 
then we have 
$$ T(\ell_k)^i (x) = \tilde T_k^i (x), \text{ for } 0 \le i < m. $$
\end{lemma}
Note that
\begin{equation*}
\frac{\lambda^*(\ell_k + s_{k+1}+1)}{\lambda^* (\ell_k)}
= \frac{\lambda_{A} (\ell_k) + \lambda_{B} (\ell_k) + \lambda_{D} (\ell_k) }{\lambda^* (\ell_k)}  = 1 - \frac{\lambda_C (\ell_k)}{\lambda^* (\ell_k)} >1 - \frac{1}{g^{2k+3}}. 
\end{equation*}
Moreover, we have
$$ \frac{F_{2^{k+2} + 2k+4}}{F_{2^{k+2} + 2k+5}} < \frac{\lambda_{\tilde A} (\ell_k)}{\lambda_{\tilde A} (\ell_k) + \lambda_{\tilde D} (\ell_k)} < \frac{F_{2^{k+2} + 2k+3}}{F_{2^{k+2} + 2k+4}}, $$ 
so $$ \left| \frac{\lambda_{\tilde A} (\ell_k)}{\lambda_{\tilde A} (\ell_k) + \lambda_{\tilde D} (\ell_k)} - \frac 1g \right | < \frac{1}{g^{2^{k+3}+4k+6}}. $$ 

Let $R_k (x)$ be the irrational rotation by $\frac{\lambda^*(\ell_k + s_{k+1} +1)}{g}$ on $[0,\lambda^*(\ell_k + s_{k+1} +1))$.

\begin{lemma}\label{ll}
For each $x \in [0, \lambda^*(\ell_k + s_{k+1} +1))$
$$ \left | \tilde T_k^i (x) - R_k^i(x) \right | < \frac{i\lambda^*(\ell_k + s_{k+1} +1)}{g^{2^{k+3}+4k+6}}.$$
\end{lemma}

By the celebrated theorem from Diophantine approximation we have
\begin{lemma}\label{lll}
For each $x \in [0, \lambda^*(\ell_k + s_{k+1} +1))$
$$ \left |  R_k^i(x) - x \right | > \frac{\lambda^*(\ell_k + s_{k+1} +1)}{2 i}.$$
\end{lemma}

\begin{proposition}
We have 
$$ \lim_{r \to 0^+} \frac{\log \tau_{r}(x)}{-\log r} = 1, \text{ a.e. } x.$$
\end{proposition}

\begin{proof}
For a general Lebesgue measure preserving transformation on the interval 
it is well known (e.g. \cite{KK}) that  
$$ \limsup_{r \to 0^+} \frac{\log \tau_{r}(x)}{-\log r} \le 1, \text{ a.e. } x.$$
We only need to show the inferior limit is not smaller than 1.

Let $\mathcal P_k$ be the partition on $[0,1)$ consisting of 
$$T^i(I_\alpha(\ell_k)), \quad 0 \le i < Q_{\alpha}(\ell_k)$$
and $P_k(x)$ be the element of $\mathcal P_k$ which contains $x$.

Fix an $\varepsilon >0$. let
$$  
E_k = \left \{ x \in [0,1) : \tau_r(x) < r^{-(1-\varepsilon)} \text{ for some } \frac{\lambda^*(\ell_{k+1})}{g^{k+1}} < r \le \frac{\lambda^*(\ell_k)}{g^k} \right \}.
$$
There are two cases : $x$ and $T^{\tau_r(x)} (x)$ are in same $\mathcal P_n$ or not.
Therefore
$$ E_k \subset F_k \cup G_k$$ 
where 
\begin{align*}
F_k &= \left \{ x : \tau_r(x) < r^{-(1-\varepsilon)}, T^{\tau_r(x)} (x) \notin P_k(x) \text{ for some } \frac{\lambda^*(\ell_{k+1})}{g^{k+1}} < r \le \frac{\lambda^*(\ell_k)}{g^k} \right \}, \\
G_k &= \left \{ x : \tau_r(x) < r^{-(1-\varepsilon)}, T^{\tau_r(x)} (x) \in P_k(x) \text{ for some } \frac{\lambda^*(\ell_{k+1})}{g^{k+1}} < r \le \frac{\lambda^*(\ell_k)}{g^k} \right \}.
\end{align*}
Clearly we have
$$ 
F_k \subset \left \{ x \in [0,1) :  \min(x-a, b-x) \le \frac{\lambda^* (\ell_k)}{g^k} \text{ if } P_k(x) = [a,b) \right \}.
$$
Since $\lambda^*(\ell_k) < g^3 \lambda_D (\ell_k) < g^3 \lambda_A (\ell_k) $,
we have 
\begin{equation}\label{fkbound}
\begin{split}
\mu(F_k) &\le Q_A(\ell_k) \frac{2\lambda^* (\ell_k)}{g^k} + 
Q_D(\ell_k) \frac{2\lambda^* (\ell_k)}{g^k} + Q_B(\ell_k) \lambda_B (\ell_k) + Q_C(\ell_k) \lambda_C (\ell_k) \\
&< Q_A(\ell_k) \frac{2 \lambda_A (\ell_k)}{g^{k-3}} + 
Q_D(\ell_k) \frac{2 \lambda_D (\ell_k)}{g^{k-3}} + 2 Q_C(\ell_k) \lambda_C (\ell_k) \\
&\le \frac{2}{g^{k-3}} + \frac{2}{g^{2k+2}}.
\end{split}
\end{equation}

Let 
$$ m = \left( \frac{g^{k+1}}{\lambda^*(\ell_{k+1})} \right)^{1-\varepsilon} \frac{1}{\min_\alpha Q_\alpha (\ell_k)}
 = \left( \frac{g^{k+1}}{\lambda^*(\ell_{k+1})} \right)^{1-\varepsilon} \frac{1}{Q_A (\ell_k)}. $$
Choose $k$ big enough to
$$\left( \frac{\lambda^*(\ell_{k+1})}{g^{k+1}} \right)^{\varepsilon} 
< \left( \frac{\lambda^*(\ell_{k})}{g^{2^{k+2}+2k+3}g^{k+1}} \right)^{\varepsilon} < \frac{1}{g^{2k+4}}. $$
Then, by Lemma~\ref{ll}, for $y \in [0, \lambda^*(\ell_k + s_{k+1} +1))$ and $0 \le i < m$ we have
\begin{equation*}
\begin{split} 
\left | \tilde T_k^i (y) - R_k^i(y) \right | &< \frac{\lambda^*(\ell_k + s_{k+1} +1)}{g^{2^{k+3}+4k+6}}\cdot m
<\frac{\lambda^*(\ell_k)}{g^{2^{k+3}+4k+6}Q_A (\ell_k)}\left( \frac{g^{k+1}}{\lambda^*(\ell_{k+1})} \right)^{1-\varepsilon} \\
&= \frac{\lambda^*(\ell_{k+1})}{g^{2^{k+3}+3k+5}} \cdot\frac {1}{\lambda^*(\ell_k)Q_A (\ell_k)} \cdot \left( \frac{\lambda^*(\ell_k)}{\lambda^*(\ell_{k+1})} \right)^2 \cdot \left( \frac{\lambda^*(\ell_{k+1})}{g^{k+1}} \right)^{\varepsilon} \\
&< \frac{\lambda^*(\ell_{k+1})}{g^{2^{k+3}+3k+5}} \cdot g \cdot g^{2^{k+3}+4k+6} \cdot \left( \frac{\lambda^*(\ell_{k+1})}{g^{k+1}} \right)^{\varepsilon}
< \frac{\lambda^*(\ell_{k+1})}{g^{k+2}}.
\end{split}\end{equation*}
By Lemma~\ref{lll} for $y \in [0, \lambda^*(\ell_k + s_{k+1} +1))$ and $0 \le i < m$
\begin{equation*}
\begin{split}
\left | R_k^i(y) - y \right | &> \frac{\lambda^*(\ell_k + s_{k+1} +1)}{2} \cdot \frac 1m
> \frac{\lambda^*(\ell_k) }{g^2} \cdot Q_A (\ell_k) \left( \frac{\lambda^*(\ell_{k+1})}{g^{k+1}} \right)^{1-\varepsilon} \\
&> \frac{1}{g^3} \left( \frac{\lambda^*(\ell_{k+1})}{g^{k+1}} \right)^{1-\varepsilon} 
= \frac{\lambda^*(\ell_{k+1})}{g^{k}}\cdot \frac{1}{g^4}\cdot\left( \frac{g^{k+1}}{\lambda^*(\ell_{k+1})} \right)^{\varepsilon}
 > \frac{\lambda^*(\ell_{k+1})}{g^{k}}.
\end{split}\end{equation*}
Therefore by Lemma~\ref{l} we have for $0 \le i < m$
$$\left | T(\ell_k)^i (y) - y \right | > \frac{\lambda^*(\ell_{k+1})}{g^{k+1}}\ \text{ for } y \in \left( I_A(\ell_k) \cup I_D(\ell_k)\right) \setminus \left( \bigcup_{i=0}^{m} T(\ell_k)^{-i} I_B(\ell_k) \right).$$
For $j \ge 0$ we can find $i < j / \min_\alpha Q_\alpha(\ell_k) < j/ Q_A (\ell_k)$ such that 
$$T(\ell_k)^i (y)= T^{j} (y).$$
Therefore, we have for $0 \le j < mQ_A(\ell_k)$
$$\left | T^j (y) - y \right | > \frac{\lambda^*(\ell_{k+1})}{g^{k+1}} \ \text{ for } y \in \left( I_A(\ell_k) \cup I_D(\ell_k)\right) \setminus \left( \bigcup_{i=0}^{m} T(\ell_k)^{-i} I_B(\ell_k) \right).$$

For each $x \in T^i(I_\alpha(\ell_k)), 0 \le i < Q_{\alpha}(\ell_k)$,
let $$\phi_k (x) = T^{-i} (x) \in I_\alpha(\ell_k) \subset [0, \lambda^*(\ell_k)).$$
If $T^r (x) \in P_k(x)$, then we have
$$T^\tau( \phi_k(x) ) - \phi_k(x) = T^\tau (x) - x .$$
Hence we have
\begin{equation*}
\begin{split}
G_k &\subset \left \{ x \in [0,1) : \phi_k(x) \notin \left( I_A(\ell_k) \cup I_D(\ell_k)\right) \setminus \left( \bigcup_{i=0}^{m} T(\ell_k)^{-i} I_B(\ell_k) \right)  \right \} \\
&= \left \{ x \in [0,1) : \phi_k(x) \in  \left( \bigcup_{i=0}^{m} T(\ell_k)^{-i} I_B(\ell_k) \right) \cup I_C (\ell_k) \right \} .
\end{split}
\end{equation*}
Therefore, we have
\begin{equation}\label{gkbound}
\begin{split}
\mu( G_k) &\le m \lambda_B(\ell_k) \max_\alpha Q_\alpha (\ell_k) + \lambda_C (\ell_k) Q_C (\ell_k) \\
&= g^{k+1} \cdot \left( \frac{\lambda^*(\ell_{k+1})}{g^{k+1}} \right)^\varepsilon \cdot \frac{\lambda_B(\ell_k) }{\lambda^*(\ell_{k+1})} \cdot \frac{Q_D (\ell_k)}{Q_A (\ell_k)} + \lambda_C (\ell_k) Q_C (\ell_k) \\
&< g^{k+1} \cdot \frac{1}{g^{2k+4}} \cdot 1 \cdot \left(g + \frac{1}{g^{2k}} \right) + \frac{1}{g^{2k+2}} < \frac{1}{g^{k+1}} + \frac{1}{g^{2k+2}}.
\end{split}
\end{equation}

From (\ref{fkbound}) and (\ref{gkbound}), the Borel-Cantelli Lemma implies that for almost every $x$, $x \in E_k$ finitely many $k$'s.
Therefore we have
$$ \liminf_{r \to 0^+} \frac{\log \tau_{r}(x)}{-\log r} \ge 1, \text{ a.e. } x.$$
\end{proof}

\section{Example with Condition (Z) without Condition (U)}\label{exm_sec2}

In this section, we discuss an example of 4 interval exchange map such that satisfies Condition (Z) but not Condition (U).

Let $T$ be the interval exchange map with the permutation data $\pi^{(0)} = \bigl( \begin{smallmatrix} A&B&D&C\\D&A&C&B \end{smallmatrix} \bigr)$
and the infinite path in the Rauzy diagram denoted by the winner of each arrow
$$ C B^3 \left(D^2 A^3 D\right)^{2^1} B  \cdot
C B^3 \left(D^2 A^3 D\right)^{2^2} B \cdots C B^3 \left(D^2 A^3 D\right)^{2^{k}} B \cdots . $$
Then there is no path of more than 3 arrows of the same winner. Thus, $T$ satisfies Condition (Z).

Let
$$\ell_k = \sum_{i=1}^k \left( 5 + 6 \cdot  2^{i} \right)
= 5 k + 12 \cdot (2^{k} - 1), \quad \ell_0 = 0.$$
Then $\gamma (\ell_{k-1}, \ell_k)$ is $C B^3 \left(D^2 A^3 D\right)^{2^{k}} B $
and 
$$
Q( \ell_{k-1}, \ell_k ) = \begin{bmatrix} 
F_{2^{k+1}+1} & F_{2^{k+1}} & F_{2^{k+1}} & F_{2^{k+1}} \\
F_{2^{k+1}+1}-1 & F_{2^{k+1}}+1 &  F_{2^{k+1}}+1 & F_{2^{k+1}} \\
F_{2^{k+1}+1}-1 & F_{2^{k+1}}+2 & F_{2^{k+1}}+3 &  F_{2^{k+1}} \\ 
F_{2^{k+1}+2}-1 & F_{2^{k+1}+1}+ 1&F_{2^{k+1}+1}+ 1 & F_{2^{k+1}+1} \end{bmatrix},
$$
where $F_n$ is the Fibonacci sequence as before.
Here, we have
$$ \|Q( \ell_{k-1}, \ell_k )\| = 9F_{2^{k+1}} + 6F_{2^{k+1}+1} + F_{2^{k+1}+2} + 6 < \frac{8g^{2^{k+1}+2}}{\sqrt 5} <  g^{2^{k+1}+5}.$$
Also we have
$$ \lambda^*(\ell_k) 
< \frac{\lambda^*(\ell_{k-1})}{F_{2^{k+1}+1}+3F_{2^{k+1}}}
< \frac{\lambda^*(\ell_{k-1})}{F_{2^{k+1}+3}}. $$
Note $T(\ell_{k}+3)$ has the same permutation data with $T(\ell_k)$, $\pi^{(\ell_{k}+3)} = \bigl( \begin{smallmatrix} A&B&D&C\\D&A&C&B \end{smallmatrix} \bigr)$.
The matrix for the path $B \left(D^2 A^3 D\right)^{2^{k+1}} B$ starting from $\bigl( \begin{smallmatrix} A&B&D&C\\D&A&C&B \end{smallmatrix} \bigr)$ is 
$$ Q(\ell_{k} + 3, \ell_{k+1} ) = \begin{bmatrix} 
F_{2^{k+2}+1} & 0 & 0 &F_{2^{k+2}} \\
F_{2^{k+2}+1}-1 & 1 & 0 &  F_{2^{k+2}} \\
F_{2^{k+2}+1}-1 & 1 & 1 & F_{2^{k+2}} \\
F_{2^{k+2}+2}-1 & 1 &  0 & F_{2^{k+2}+1} \end{bmatrix}.$$

Since $\lambda(\ell_{k+1}) Q(\ell_{k} +3, \ell_{k+1} ) = \lambda(\ell_k + 3)$, 
length data $\lambda(\ell_k + 3)$ is a vector in the simplex with the vertexes
$$
\lambda^*(\ell_{k+1}) \begin{bmatrix} F_{2^{k+2}+1} & 0 & 0 & F_{2^{k+2}} \end{bmatrix}, 
\lambda^*(\ell_{k+1}) \begin{bmatrix} F_{2^{k+2}+1}-1& 1& 0 & F_{2^{k+2}} \end{bmatrix}, $$
$$
\lambda^*(\ell_{k+1}) \begin{bmatrix} F_{2^{k+2}+1}-1 & 1 & 1 & F_{2^{k+2}} \end{bmatrix}, 
\lambda^*(\ell_{k+1}) \begin{bmatrix}F_{2^{k+2}+2}-1 & 1 & 0 & F_{2^{k+2}+1}\end{bmatrix}. 
$$
Therefore we have
$$ 0 < \lambda_B (\ell_k +3) < \lambda^*(\ell_{k+1})$$
and for all $x \in I_C(\ell_k + 3 )$  we have 
\begin{align*}
| T(\ell_k + 3 ) (x) - x | &= \lambda_B (\ell_k +3) < \lambda^*(\ell_{k+1}) < \frac{\lambda^*(\ell_0)}{F_{2^{k+2}+2} F_{2^{k+1}+3} \cdots F_{2^{2}+3}}\\
&< \frac{5^{(k+1)/2}}{g^{2^{k+2}+2} g^{2^{k+1}+3} \cdots g^{2^{2}+3}}
= \frac{5^{(k+1)/2}}{g^{ 2^{k+3} +3k - 2}}
< \frac{1}{g^{ 2^{k+3}+k-4}}.
\end{align*}

Since 
$$ \begin{bmatrix} 1&0&0&0\\0&1&1&0\\0&1&2&0\\0&1&1&1 \end{bmatrix} Q (\ell_k) = Q (\ell_k +3 ),$$
we have
\begin{align*}
Q_C (\ell_k +3 ) &= Q_B (\ell_k) + 2 Q_C (\ell_k)
\le 2 \| Q(\ell_k) \| \le 2 \| Q(\ell_{k-1},\ell_k) \cdots Q(\ell_0, \ell_1) \|   \\
&\le  2 \| Q(\ell_{k-1},\ell_k) \|  \cdots \| Q(\ell_0, \ell_1) \|
< 2 g^{2^{k+1}+5} \cdots g^{2^2+5}< 2 g^{2^{k+2} + 5k} .
\end{align*}
Thus, put $r = \lambda_B (\ell_k +3)$. Then if $k \ge 4$, we have for $x \in I_C(\ell_k + 3 )$
\begin{align*}
\frac{\log \tau_r (x)}{-\log r} &< \frac{\log Q_C (\ell_k +3 )}{- \log \lambda_B (\ell_k +3)} < \frac{(2^{k+2} + 5k) \log g + \log 2}{( 2^{k+3} + k - 4)\log g} \\
&< \frac{1 + 5 \cdot k \cdot 2^{-k-2} + 2^{-k-1}}{2}  
< \frac 34.
\end{align*}
Hence, $\frac{\log \tau_r (x)}{-\log r}$ does not converges to 1 uniformly.

\end{document}